\documentclass[11pt,a4paper]{amsart}
\usepackage[utf8]{inputenc}
\usepackage[a4paper,margin=2.5cm]{geometry}
\usepackage{amsmath,amssymb,amsthm,mathtools,abraces}
\usepackage{stmaryrd}
\usepackage{color,xcolor}
\usepackage[hypertexnames=false,hyperfootnotes=false,colorlinks=true,linkcolor=blue,%
citecolor=purple,filecolor=magenta,urlcolor=cyan,unicode,linktocpage=true,pagebackref=false]{hyperref}
\usepackage{yhmath}
\usepackage{verbatim}

\newcommand\be{\begin{equation}}
\newcommand\ee{\end{equation}}
\newcommand\nn{\nonumber}

\newcommand\p{\partial}

\newcommand{\<}{\left <}
\renewcommand{\>}{\right >}

\DeclareMathOperator{\GL}{GL}
\DeclareMathOperator{\res}{Res}
\DeclareMathOperator{\odd}{odd}
\DeclareMathOperator{\gl}{\mathfrak{gl}}
\DeclareMathOperator\arctanh{arctanh}

\makeatletter
\@namedef{subjclassname@2020}{%
  \textup{2020} Mathematics Subject Classification}
\makeatother

\newtheorem{theorem}{Theorem}
\newtheorem{lemma}{Lemma}[section]
\newtheorem{proposition}[lemma]{Proposition}
\newtheorem{corollary}[lemma]{Corollary}
\newtheorem{remark}{Remark}[section]

\newtheorem{conjecture}{Conjecture}[section]
\newtheorem*{theorem*}{Theorem}

\newtheorem{definition}{Definition}

\numberwithin{equation}{section}

\title[KP integrability of triple Hodge integrals. III. ]{
KP integrability of triple Hodge integrals. III. \\
Cut-and-join description, KdV reduction, and topological recursions}

\author{Alexander Alexandrov}

\address[]{IBS Center for Geometry and Physics,
	Pohang University of Science and Technology (POSTECH),
	77 Cheongam-ro, Nam-gu, Pohang, Gyeongbuk, 37673, Korea
}

\email{ {\tt alexandrovsash at gmail.com}}


\begin{document}

\begin{abstract} 
In this paper, we continue our investigation of the triple Hodge integrals satisfying the Calabi--Yau condition. For the tau-functions, which generate these integrals, we derive the complete families of the Heisenberg--Virasoro constraints. We also construct several equivalent versions of the cut-and-join operators. These operators describe the algebraic version of topological recursion. For the specific values of parameters associated with the KdV reduction, we prove that these tau-functions are equal to the generating functions of intersection numbers of $\psi$ and $\kappa$ classes. We interpret this relation as a symplectic invariance of the  Chekhov--Eynard--Orantin topological recursion and prove this recursion for the general $\Theta$-case. 
\end{abstract}

\maketitle


\tableofcontents


\def\thefootnote{\arabic{footnote}}

\section*{Introduction}
\setcounter{equation}{0}

\subsection{Algebraic topological recursion}

The Chekhov--Eynard--Orantin topological recursion  is an effective and universal tool in mathematical physics, enumerative geometry, and combinatorics \cite{EO,EO1}. For the basic models, this recursion provides the solution to the Virasoro constraints. However, there are other ways to solve them, and sometimes they suit better for the particular needs. For instance, the Kontsevich--Witten and Br\'ezin--Gross--Witten tau-functions, which are the basic elements of the Chekhov--Eynard--Orantin topological recursion, can be described by explicit formulas in terms of the second order differential operators \cite{KSS,KS2}.

This description results from a solution of an equation given by the combination of the Virasoro constraints. It was recently noted \cite{MMMR} that
the partition functions of some interesting matrix models  are completely specified by a unique equation of this type. In this paper, we prove a vast generalization of this observation. Namely, we prove that the solution of a rather general equation, which includes the Euler operator as the leading term, exists and is unique. This solution can be constructed recursively.

For the generating functions with the natural topological expansion this recursive procedure may have a transparent topological meaning. In this case the recursion 
can be interpreted as the {\em algebraic topological recursion}. As this recursion is given by the differential operators (also known as {\em cut-and-join operators}), it is rather simple computationally. 
It is not known in general how to construct the equation leading to the cut-and-join solution. In some cases this equation can be constructed from the Virasoro and W-constrains, satisfied by the generating functions. One of the main goals of this paper is to provide the cut-and-join description for an interesting family of the enumerative geometry invariants, namely for the triple Hodge integrals satisfying the Calabi--Yau condition.

\subsection{Intersection theory and KP integrability}

Let $\overline {\mathcal M}_{g,n}$ be the Deligne--Mumford compactification of the moduli space of all compact Riemann surfaces of genus~$g$ with~$n$ distinct marked points. It is empty unless the stability condition
\begin{gather}
2g-2+n>0
\end{gather}
is obeyed. 

On $\overline {\mathcal M}_{g,n}$ consider four families of cohomological classes: $\psi_k$, $\kappa_k$, $\lambda_k$, and $\Theta_{g,n}$. 
For each marking index~$i$ we consider the cotangent line bundle ${\mathbb{L}}_i \rightarrow \overline{\mathcal{M}}_{g,n}$, whose fiber over a point $[\Sigma,z_1,\ldots,z_n]\in \overline{\mathcal{M}}_{g,n}$ is the complex cotangent space $T_{z_i}^*\Sigma$ of $\Sigma$ at $z_i$. Then $\psi_i\in H^2(\overline{\mathcal{M}}_{g,n},\mathbb{Q})$ is the first Chern class of ${\mathbb{L}}_i$. With the forgetful map for the marked point $n+1$, $\pi: \overline{\mathcal M}_{g,n+1} \rightarrow \overline{\mathcal M}_{g,n} $, we define the Miller--Morita--Mumford tautological classes \cite{Mumford}, $\kappa_k:= \pi_* \psi_{n+1}^{k+1} \in H^{2k}(\overline{\mathcal{M}}_{g,n},\mathbb{Q})$. For the Hodge bundle $\mathbb E$, a rank $g$ vector bundle over $\overline{\mathcal{M}}_{g,n}$, we consider the $i$th Chern classes $\lambda_i=c_i({\mathbb E})$ and their generating function $\Lambda_g(u)=\sum_{i=0}^g u^i \lambda_i$. Finally, following Norbury \cite{Norb}, we introduce $\Theta$-classes, $\Theta_{g,n}\in H^{4g-4+2n}(\overline{\mathcal{M}}_{g,n})$, which are also related to the super Riemann surfaces \cite{NorbS}. We refer the reader to \cite{Norb,NP1,NorbS} for a detailed description of $\Theta$-classes. 

By the Kontsevich--Witten theorem \cite{Kon92,Wit91} the generating function of the intersection numbers 
\be
\int_{\overline{\mathcal{M}}_{g,n}} \psi_1^{a_1} \psi_2^{a_2} \cdots \psi_n^{a_n}
\ee
is a tau-function of the KdV hierarchy. For their $\Theta$-version,
\be
\int_{\overline{\mathcal{M}}_{g,n}}\Theta_{g,n} \psi_1^{a_1} \psi_2^{a_2} \cdots \psi_n^{a_n},
\ee
KdV integrability of the generating function was established by Norbury \cite{Norb}. Relation between the intersection theory over the moduli spaces and solitonic integrable hierarchies is rather mysterious, and only a few instances of this relation are known so far.

The Kontsevich--Witten (KW) and Br\'ezin--Gross--Witten (BGW) tau-functions, corresponding to the intersection numbers of $\psi$ and $\Theta$ classes, play very important role in the modern mathematical physics and enumerative geometry. In particular, they are the main ingredients of such universal constructions as the Chekhov--Eynard--Orantin topological recursion \cite{EO,EO1} and the Givental decomposition \cite{Giv1,Teleman}. These tau-functions share a lot of common properties, in particular they satisfy the similar looking Virasoro constraints \cite{Fuku,DVV,GN}, and it is convenient to consider them simultaneously. The Virasoro constraints allow us  \cite{KSS,KS2} to derive the cut-and-join description of these tau-functions and to prove the algebraic version of topological recursion, associated with the expansion with respect to the negative Euler characteristics of the punctured Riemann surface. We claim that it is natural to consider all types of intersection numbers and their $\Theta$-versions simultaneously.  In many respects $\Theta$-versions are a bit simpler, perhaps because of the supersymmetry behind them.

In this paper, we investigate triple Hodge integrals satisfying the Calabi--Yau condition,
\be\label{Intin}
\int_{\overline{\mathcal{M}}_{g,n}} \Theta_{g,n}^{1-\alpha} \Lambda_g (-q) \Lambda_g (-p) \Lambda_g \left(\frac{pq}{p+q}\right)\psi_1^{a_1} \psi_2^{a_2} \cdots \psi_n^{a_n}
\ee
for $\alpha=0$ and $\alpha=1$. For $\alpha=1$ these integrals appear at the localization techniques for the Gromov--Witten invariants on the three-dimensional target manifolds \cite{GP}.
They constitute the topological vertex, which provides a diagrammatic method to compute topological string amplitudes on non-compact toric Calabi--Yau three-folds \cite{tv,tvm}. String theory interpretation of these integrals for $\alpha=0$ is not known yet.

It appears that the generating functions of these intersection numbers have a surprisingly simple description in terms of integrable hierarchies. Namely, 
these generating functions are tau-functions of the KP hierarchy \cite{H3_1}. These tau-functions are certain deformations of the KW and BGW tau-functions. This generalizes the result of Kazarian for the linear Hodge integrals \cite{Kaza}.  

In this paper, we continue our investigation of the tau-functions for the triple Hodge integrals satisfying the Calabi--Yau condition. In particular, for these tau-functions using their relation to the KW and BGW tau-functions, we derive a complete set of the Heisenberg--Virasoro constraints, see Theorem \ref{propp}. These linear constrains completely specify a formal series solution. However, they are not sufficient to construct the cut-and-join description. In Theorem \ref{T1} we construct this description using the KP hierarchy group of symmetries. Resulting algebraic topological recursion, given by the third order differential operators, is very convenient computationally. We use it to compute the first few terms of the topological expansion of the tau-functions for the triple Hodge integrals satisfying the Calabi--Yau condition, see Appendix \ref{AA}.

\subsection{KdV reduction}
Let us also consider the intersection numbers of $\psi$, $\kappa$ and $\Theta$ classes 
\be\label{kain}
\int_{\overline{\mathcal{M}}_{g,n}} \Theta_{g,n}^{1-\alpha} e^{\sum_{j=1}^\infty s_j \kappa_j}\psi_1^{a_1} \psi_2^{a_2} \cdots \psi_n^{a_n}.
\ee
According to \cite{KMZ,MZ,NorbS} the generating functions of these intersection numbers can be identified with the KW (for $\alpha=1$) and BGW (for $\alpha=0$) tau-functions with shifted arguments.

For $p=-2q$ the generating functions of the triple Hodge integrals satisfying the Calabi--Yau condition \eqref{Intin} are the tau-functions of the KdV hierarchy \cite{H3_2}. In Theorem \ref{taueq} we prove that in this case the tau-functions coincide with the generating functions of intersection numbers \eqref{kain} for the specific values of $s_k$ variables. The proof is based on the Virasoro constraints and the explicit form of the group element of the KP hierarchy symmetry. A geometric interpretation of this relation is not clear yet. 

This theorem allows us to identify the intersection numbers \eqref{Intin} with \eqref{kain} for arbitrary genus $g$ and number of marked points $n$ and specific values of parameters. We interpret this identification in terms of the symplectic invariance of the Chekhov--Eynard--Orantin topological recursion. 
For $\alpha=1$, where the topological recursion for \eqref{Intin} is known \cite{BM,Chen,Zhou1}, we use the symplectic invariance to find an independent proof of the relations between the intersection numbers.  For $\alpha=0$ we prove the topological recursion for the general triple $\Theta$-Hodge integrals satisfying the Calabi--Yau condition.

\addtocontents{toc}{\protect\setcounter{tocdepth}{1}}
\subsection*{Organization of the paper}

In Section \ref{S1} we remind the reader some basic properties of the Heisenberg--Virasoro symmetry algebra of the KP hierarchy. Section \ref{S2} is devoted to the origin of the cut-and-join description. In particular, we derive it as a solution of the differential equation given by a rather general deformation of the Euler operator and its relation to the topological expansion. In Section \ref{S3} we remind the main properties of the KW and BGW tau-functions, and investigate how these properties transform under the action of the Virasoro group. In Section \ref{H3} we construct the Heisenberg--Virasoro constraints and the cut-and-join operators for the tau-functions, which describe the triple Hodge integrals satisfying the Calabi--Yau condition. Section \ref{S5} is devoted to the reduction to the KdV hierarchy for the specific values of the parameters. We prove that in this case the tau-functions coincide with the tau-functions for the intersection numbers containing the Miller--Morita--Mumford classes. In Section \ref{S6} we describe this relation using symplectic invariance of the Chekhov--Eynard--Orantin topological recursion and prove this recursion for the general triple $\Theta$-Hodge integrals satisfying the Calabi--Yau condition.

\subsection*{Acknowledgments}

This work was supported by the Institute for Basic Science (IBS-R003-D1). 

\subsection*{Conflict of interest statement}

We have no conflicts of interest to disclose.

\section{Symmetries of KP hierarchy and Heisenberg--Virasoro algebra}\label{S1}
\addtocontents{toc}{\protect\setcounter{tocdepth}{2}}

It is well known that a certain central extension of the algebra $\gl(\infty)$ and corresponding group $\GL(\infty)$ act on the space of KP tau-functions \cite{F,MJ}. From the boson-fermion correspondence we also know how to identify the operators, acting on the tau-functions with the operators, acting on the Sato Grassmannian. Often the Sato Grassmannian parametrization is more convenient, and we will use it for computations.
 In this section, we remind the reader the properties of the important Heisenberg--Virasoro subalgebra of $\gl(\infty)$; for more details see, e.g., \cite{A1}.

Consider the space ${\mathbb C}[\![{\bf t}]\!]$ of formal power series of infinitely many variables ${\bf t}=\{ t_1, t_2,\dots\}$. We will denote by $\widehat{\dots}$ differential operators acting on this space. 
The {\em Heisenberg--Virasoro subalgebra} ${\mathcal L}$ of $\gl(\infty)$ is generated by the operators
\be
\widehat{J}_k =
\begin{cases}
\displaystyle{\frac{\p}{\p t_k} \,\,\,\,\,\,\,\,\,\,\,\, \mathrm{for} \quad k>0},\\[2pt]
\displaystyle{0}\,\,\,\,\,\,\,\,\,\,\,\,\,\,\,\,\,\,\, \mathrm{for} \quad k=0,\\[2pt]
\displaystyle{-kt_{-k} \,\,\,\,\,\mathrm{for} \quad k<0,}
\end{cases}
\ee
the unit, and the Virasoro operators
\be
\label{virfull}
\widehat{L}_m=\frac{1}{2} \sum_{a+b=-m}a b t_a t_b+ \sum_{k=1}^\infty k t_k \frac{\p}{\p t_{k+m}}+\frac{1}{2} \sum_{a+b=m} \frac{\p^2}{\p t_a \p t_b}.
\ee
These operators satisfy the commutation relations
\begin{align}\label{comre}
\left[\widehat{J}_k,\widehat{J}_m\right]&=k \delta_{k,-m},\nn\\
\left[\widehat{L}_k,\widehat{J}_m\right]&=-m \widehat{J}_{k+m},\\
\left[\widehat{L}_k,\widehat{L}_m\right]&=(k-m)\widehat{L}_{k+m}+\frac{1}{12}\delta_{k,-m}(k^3-k)\nn
\end{align}
and generate a Lie algebra.
We decompose ${\mathcal L}$ as ${\mathcal L}={\mathcal L}_+\oplus  {\mathcal L}_0 \oplus {\mathcal L}_-$, where ${\mathcal L}_+$ (${\mathcal L}_-$) are generated by the operators $\widehat{L}_k$ and $\widehat{J}_k$ with positive (negative) $k$, and ${\mathcal L}_0$ is generated by $\widehat{L}_0$ and a unit.

The Heisenberg--Virasoro group ${\mathcal V}$ is generated by the operators $\widehat{J}_k$,  $\widehat{L}_k$ and a unit,
\be
{\mathcal V}:=\left.\{C \, e^{\sum a_k \widehat{J}_k +b_{k}\widehat{L}_k} \right| a_k,b_k,C \in {\mathbb C}\}.
\ee
For this group we also have a decomposition, ${\mathcal V}={\mathcal V}_+\oplus  {\mathcal V}_0 \oplus {\mathcal V}_-$. This group acts on the space of the tau-functions of the KP hierarchy.

The Heisenberg--Virasoro algebra \eqref{comre} is a central extension of the algebra of operators, acting on the Sato Grassmannian and 
generated by
\be\label{virw}
{\mathtt j}_m=z^m,\,\,\,\,\,\,\,\,\,
{\mathtt l}_m=-z^m\left(z\frac{\p}{\p z}+\frac{m+1}{2}\right).
\ee
Operators ${\mathtt l}_m$ generate the Witt algebra.
These operators obey the commutation relations (\ref{comre}) with omitted central term. Thanks to this relation between the commutation relations it is convenient to use the operators ${\mathtt j}_m$ and ${\mathtt l}_m$ for computations in the group ${\mathcal V}$.

The action of the ${\mathcal V}_0$ and  ${\mathcal V}_-$ groups on the space of the functions of ${\bf t}$ is rather simple -- it is equivalent to the linear change of the variables supplemented by the 
multiplication by a function. Below we will mostly focus on the group ${\mathcal V}_+$.


\section{Algebraic topological recursion and cut-and-join operators}\label{S2}

In this section, we describe how a single equation of a certain type satisfied by the generating functions can lead to the recursion relations and the cut-and-join description which specify a unique solution. Moreover, 
for the generating functions of the enumerative geometry and mathematical physics this recursion can have a natural topological interpretation.

\subsection{General cut-and-join description}

Let us consider (in general, infinitely many) formal variables ${\bf q}=\{q_1,q_2,q_3,\dots\}$. We associate with them positive integer degrees $a_k\in {\mathbb Z}_{>0}$. Consider corresponding Euler operator 
\be
E=\sum_k a_k q_k \frac{\p}{\p q_k}.
\ee
We assume that there are only finitely many variables with any finite degree, that is, $\# q_k (a_k<n)<\infty $ for any $0<n<\infty$. We associate the degree $-a_k$ to the operators $\frac{\p}{\p q_k}$ and consider a differential equation
\be\label{Weq1}
\left(E -\sum_{k=1}^\infty k V_{k}\right) \cdot Z({\bf q})=0,
\ee
where ${V}_k$ are some differential operators of degree $k$ with coefficients given by formal series in variables ${\bf q}$.  It appears that such an equation always has a unique solution.
\begin{proposition}\label{Prop1}
Equation \eqref{Weq1} has a unique solution in ${\mathbb C}[\![{\bf q}]\!]$ up to a normalization.
\end{proposition}
\begin{proof}
Let us find this solution explicitly. Consider the expansion 
\be\label{Zexp}
Z({\bf q})=\sum_{k=0}^\infty Z_{k}({\bf q}),
\ee
where $Z_0\in{\mathbb C}$ is a constant, and the degree of $Z_k$ is $k$, that is, $E\cdot Z_k = k Z_k$. Then the term of degree $k$ on the left hand side of \eqref{Weq1} should vanish,
\be
E \cdot Z_{k} -\sum_{j=0}^{k-1}(k-j) {V}_{k-j} \cdot Z_{j}=0.
\ee
Since $E\cdot Z_{k} =k Z_{k}$, the unique solution of \eqref{Weq1} is given by the recursion
\be\label{eqrec}
Z_{k}=\frac{1}{k} \sum_{j=0}^{k-1}(k-j) {V}_{k-j} Z_{j}.
\ee
Because of the finiteness of the variables with a given finite degree, all $Z_{j}$ are polynomials, and the action on them of all finite degree operators is well defined.
Arbitrary constant $Z_0$ fixes the total normalization of the solution, which completes the proof.
\end{proof}
Equation \eqref{eqrec} provides a recursion which completely specifies the solution for a given initial condition $Z_0=Z({\bf 0})$ which can be fixed by $Z({\bf 0})=1$. This solution is described by the operators $V_k$. There are several possibilities depending on the relations between the operators $V_k$.
\begin{itemize}
\item If ${V}_k=\delta_{k,r} {V}_r$ for some $r>0$,  we have an explicit solution
\be\label{symsol}
Z({\bf q})=\exp{\left({V}_r\right)}\cdot 1.
\ee
\item In general, if the operators ${V}_k$ constitute a commutative family, $\left[{V}_k,{V}_m\right]=0$ for all $k,m>0$, then
\be\label{symsol1}
Z({\bf q})=\exp\left({\sum_{k=1}^\infty {V}_k}\right)\cdot 1.
\ee
If $V_k=0$ for all $k>r$ for some $r$, then the sum in the right hand side of \eqref{eqrec} contains at most $r$ terms.
\item If the operators $V_k$ do not commute, then the solution is given by the ordered exponential, which we denote by $``\exp\left(\dots\right)"$,
\be\label{fors}
Z({\bf q})= ``\exp\left({\sum_{k=1}^\infty{V}_k}\right)" \cdot 1.
\ee
\end{itemize}
\begin{remark}
Equation of the form \eqref{Weq1} is never unique. Therefore, in general it is not clear how to choose the best representative of equations of this type and the corresponding solution \eqref{fors}. At least for some interesting models arising in enumerative geometry and matrix models theory the generating functions can be represented by \eqref{symsol}. 
 \end{remark}

In particular, one can consider the functions of the variables $t_k$.
We introduce a ${\mathbb Z}$ grading on functions and differential operators on this space by assigning 
\be\label{deg}
\deg t_k =k,\quad\quad \deg \frac{\p}{\p t_k} =-k.
\ee
Then $\deg \widehat{L}_k =\deg  \widehat{J}_k  =- k$ and degree of the central element of the Heisenberg--Virasoro algebra is zero. 
The Euler operator $E=\sum_{k=1}^\infty k t_k \frac{\p}{\p t_k}$ coincides with $\widehat{L}_0$.
Let us consider a differential equation
\be\label{Weq}
\left(\widehat {L}_0 -\sum_{k=1}^\infty k  \widehat{V}_{k}\right) \cdot Z({\bf t})=0,
\ee
where $\widehat{V}_k$ are some differential operators (not necessarily from the $\gl(\infty)$ algebra associated with the KP hierarchy) with $\deg \widehat{V}_k=k$. From Proposition \ref{Prop1} it follows that such an equation always have a unique solution.
The following simple statement is a generalization of the observation of \cite{MMMR}: 
\begin{corollary}
Equation \eqref{Weq} has a unique solution in ${\mathbb C}[\![{\bf t}]\!]$ up to a normalization.
\end{corollary}

\begin{remark}
For the KP hierarchy one can construct the equation of the form \eqref{Weq} using the product of the canonical Kac--Schwarz operators \cite{H3_2}. In this case the constructed operators $ \widehat{V}_{k}$ belong to the $\gl(\infty)$ algebra of KP symmetries. However, already for the Kontsevich--Witten and Br\'ezin--Gross--Witten tau-functions the answer, provided by this construction, does not coincide with the simplest description obtained from the solution of the Virasoro constraints, see Section \ref{SHV}.
\end{remark}



\subsection{Algebraic topological recursion}
Generating functions in enumerative geometry and mathematical physics often depend on a parameter $\hbar$ associated to some {\em topological expansion}.
Often it is related to the Euler characteristics $\chi$ of certain Riemann surface with punctures, in this case the topological expansion is given by $Z=\exp \sum_\chi \hbar^{-\chi} F_\chi $, or
\be\label{topexp}
Z=\sum_{k=0}^\infty \hbar^k Z^{(k)}.
\ee
Assume that the generating function $Z$ is described by the equation \eqref{Weq}, the more general equation \eqref{Weq1} can be considered similarly.
Let the generating function $Z$ also satisfy the dimension constraint
\be\label{homog}
a \hbar \frac{\p}{\p \hbar} Z=  \widehat {L}_0 \cdot Z
\ee
for some $a\in{\mathbb Z}_{>0}$. 
\begin{remark}
The constraint \eqref{homog} makes the parameter $\hbar$ redundant, however, it is convenient to use it. 
\end{remark}
Assume that the operators in \eqref{Weq} are of the form $\widehat{V}_k=\hbar^{\frac{k}{a} } \widehat{W}_k$, where $ \widehat{W}_k$ do not depend on $\hbar$. Then they are nontrivial only if $k/a\in {\mathbb Z}$, and the equation \eqref{Weq} leads to
\be\label{gene}
a \hbar \frac{\p}{\p \hbar} Z= \sum_{k=1}^\infty k a \,\hbar^{k} \widehat{W}_{ka} \cdot Z.
\ee
\begin{proposition}\label{Prop2}
For any initial condition  $Z^{(0)}$ if a solution to the equation \eqref{gene} of the form \eqref{topexp} exists, it is unique.
\end{proposition}
\begin{proof}
Indeed, if we substitute \eqref{topexp} into \eqref{gene}, we get a recursion
\be\label{TRR}
{Z}^{(k)}=\frac{1}{k} \sum_{j=0}^{k-1}  (k-j) \widehat{W}_{a(k-j)}\cdot {Z}^{(j)},
\ee
which allows us to find all ${Z}^{(k)}$ recursively starting from the initial condition $Z^{(0)}$.
\end{proof}
Note that we do not use Proposition \ref{Prop1}, therefore, Proposition \ref{Prop2} is valid even if the generating function does not satisfy any equation \eqref{Weq1}.
However, if this equation is satisfied, then $Z^{(0)}$ is a constant and the solution always exists. 
 
Equation of the form \eqref{TRR} describes the {\em algebraic topological recursion}.
The solution is given by ordered exponential
\be
Z= ``\exp\left({\sum_{k=1}^\infty \hbar^k \widehat{W}_{ka}}\right)" \cdot Z^{(0)},
\ee
which simplifies as in \eqref{symsol} or \eqref{symsol1} for the particular cases. Operators $\widehat{W}_k$ describe the change of topology, that's why below slightly abusing notation we will call them the {\em cut-and-join} operators. 

In the simplest case $ \widehat{W}_{ka} =\delta_{k,1} \widehat{{W}}_a$, and the generating function is given by
\be\label{ananeq}
Z= \exp\left({  \hbar  \widehat{W}_{a}}\right) \cdot Z^{(0)}.
\ee
In this case the algebraic topological recursion have a particularly simple form
\be
Z^{(k)}=\frac{1}{k}  \widehat{W}_{a} \cdot Z^{(k-1)}.
\ee

Let us list some enumerative geometry problems (and associated matrix models), which can be described by the cut-and-join operators 
\begin{itemize}

\item Fat graphs or maps (Hermitian matrix model),  \cite{MS};

\item Intersection theory of $\psi$ classes on the moduli spaces (Kontsevich model), \cite{KSS};

\item Intersection theory of $\psi$ and Norbury's classes on the moduli spaces (Br\'ezin--Gross--Witten model),  \cite{KS2};

\item Weighted Hurwitz numbers (certain multi-matrix models),  \cite{GJ} for ordinary Hurwitz numbers, \cite{ACEHf} and references therein for the weighted Hurwitz numbers; 

\item Open intersection theory (Kontsevich--Penner model), \cite{open};

\item Intersection theory  of $\psi$ and  Witten's classes (Generalized Kontsevich model),  \cite{Zhou,M3}.

\end{itemize}
The last two examples are described by the non-commuting families of the cut-and-join operators.

The Chekhov--Eynard--Orantin topological recursion \cite{EO,EO1} also describes a recursive procedure in the negative Euler characteristics. However, many properties of this recursion, briefly outlined in Section \ref{S6}, are essentially different from those of the algebraic topological recursion \eqref{TRR}.

The algebraic recursion \eqref{TRR} is given by differential operators with polynomial coefficients,  while the Chekhov--Eynard--Orantin recursion is given in terms of the residues on the spectral curve.
The topological recursion \eqref{TRR} is formulated in terms of the coefficients of the generating function $Z$, or non-connected correlation functions, while the Chekhov--Eynard--Orantin topological recursion is naturally formulated in terms of connected correlation functions, or the coefficients of expansion of $\log Z$. The recursion \eqref{TRR} is linear in the non-connected contributions, while the Chekhov--Eynard--Orantin recursion in quadratic in connected correlation functions.

The topological recursion 
\eqref{TRR} is often very convenient for computations.  It would be interesting to find a relation between the algebraic topological recursion, described in this paper, and geometrical topological recursion of Chekhov--Eynard--Orantin. In particular, there should be a way to construct the cut-and-join operators from the topological recursion data. We believe that it is given by some version of the inverse Laplace transform \cite{EMS}.

\begin{conjecture}
The Chekhov--Eynard--Orantin topological recursion with simple ramification points can always be described by the cut-and-join formula \eqref{ananeq} with operator $ \widehat{W}_{a}$  cubic in $\widehat{J}_k$. Higher ramification points correspond to the higher cut-and-join operators. 
\end{conjecture}
It is also unclear how to directely relate the cut-and-join description to the matrix integrals.

If two generating functions $Z$ and ${Z}^*$ are related to each other by 
\be\label{aneq}
{Z}^*= \widehat{U} \cdot Z,
\ee
where the operator $\widehat{U}$ does not depend on $\hbar$, then their cut-and-join operators are also related by conjugation.
Namely, if $Z$ satisfies \eqref{gene}, then
\be
{Z}^*= ``\exp\left({\sum_{k=1}^\infty \hbar^k \widehat{W}_{ka}^*}\right)" \cdot Z^{(0)*},
\ee
where $ \widehat{W}_k^* =\widehat{U}\,  \widehat{W}_k \, \widehat{U}^{-1}$ and $Z^{(0)*}=\widehat{U}\cdot Z^{(0)}$ are independent of $\hbar$. For the case \eqref{ananeq} we have
\be\label{conc}
{Z}^*= \exp\left({ \hbar \widehat{W}_{a}^*}\right) \cdot Z^{(0)*}.
\ee
For a more general case with $\widehat{U}$ dependent on $\hbar$ sometimes it is also possible to relate the cut-and-join descriptions.

In Sections \ref{S3} and \ref{H3} we consider the generating functions of the intersection numbers on the moduli spaces with a natural topological expansion. Namely, the KW and BGW tau-functions satisfy \eqref{homog} for $a=3$ and $a=1$ respectively. From the Virasoro constraints it follows that they are given by a cut-and-join description of the form \eqref{ananeq}, see \eqref{cajkdv} below. The generating functions of the triple Hodge integrals satisfying the Calabi--Yau condition satisfy the dimension constraints given by Proposition \ref{prop1}. For the corresponding Euler operator we do not have a description of the form \eqref{Weq1}. However, these tau-functions are related to the KW and BGW tau-functions by \eqref{aneq}, where $\widehat{U}$ are the elements of the Virasoro group, see \eqref{grel}. Using the conjugation, we derive the cut-and-join description \eqref{conc} for these generating functions.


\section{Cut-and-join operators for KW and BGW tau-functions}\label{S3}


\subsection{Intersection numbers}\label{S31}

Denote by $\overline {\mathcal M}_{g,n}$ the Deligne--Mumford compactification of the moduli space of all compact Riemann surfaces of genus~$g$ with~$n$ distinct marked points. It is a non-singular complex orbifold of dimension~$3g-3+n$. It is empty unless the stability condition
\begin{gather}\label{stability}
2g-2+n>0
\end{gather}
is satisfied.

For each marking index~$i$ consider the cotangent line bundle ${\mathbb{L}}_i \rightarrow \overline{\mathcal{M}}_{g,n}$, whose fiber over a point $[\Sigma,z_1,\ldots,z_n]\in \overline{\mathcal{M}}_{g,n}$ is the complex cotangent space $T_{z_i}^*\Sigma$ of $\Sigma$ at $z_i$. Let $\psi_i\in H^2(\overline{\mathcal{M}}_{g,n},\mathbb{Q})$ denote the first Chern class of ${\mathbb{L}}_i$. We consider the intersection numbers
\begin{gather}\label{eq:products}
\<\tau_{a_1} \tau_{a_2} \cdots \tau_{a_n}\>_g=\int_{\overline{\mathcal{M}}_{g,n}} \psi_1^{a_1} \psi_2^{a_2} \cdots \psi_n^{a_n}.
\end{gather}
The integral on the right hand side of~\eqref{eq:products} vanishes unless the stability condition~\eqref{stability} is obeyed, all  $a_i$ are non-negative integers, and the dimension constraint 
\be\label{d1}
3g-3+n=\sum_{i=1}^n a_i
\ee 
holds. Let $T_i$, $i\geq 0$, be formal variables and consider
\be
\tau_{KW}=\exp\left(\sum_{g=0}^\infty \sum_{n=0}^\infty \hbar^{2g-2+n}F_{g,n}\right),
\ee
where
\be
F_{g,n}=\sum_{a_1,\ldots,a_n\ge 0}\<\tau_{a_1}\tau_{a_2}\cdots\tau_{a_n}\>_g\frac{\prod T_{a_i}}{n!}.
\ee
Witten's conjecture \cite{Wit91}, proved by Kontsevich \cite{Kon92}, states that the partition function~$\tau_{KW}$ becomes a tau-function of the KdV hierarchy after the change of variables~$T_k=(2k+1)!!t_{2k+1}$. Integrability of the Kontsevich--Witten tau-function immediately follows from Kontsevich's matrix integral representation, for more details see e.g. \cite{H3_2} and references therein.

The Br\'ezin--Gross--Witten tau-function is another solution of the KdV hierarchy, which is closely related to the moduli spaces. Namely, this tau-function governs the intersection theory with the insertions of the fascinating Norbury's $\Theta$-classes. 
Norbury's $\Theta$-classes are the cohomology classes, $\Theta_{g,n}\in H^{4g-4+2n}(\overline{\mathcal{M}}_{g,n})$, which are also related to the super Riemann surfaces \cite{NorbS}. We refer the reader to \cite{Norb,NP1} for a detailed description. Let us consider the intersection numbers
\be
\<\tau_{a_1} \tau_{a_2} \cdots \tau_{a_n}\>_g^\Theta= \int_{\overline{\mathcal{M}}_{g,n}}\Theta_{g,n} \psi_1^{a_1} \psi_2^{a_2} \cdots \psi_n^{a_n}. 
\ee
\begin{remark}
For the $\Theta$-classes, higher degrees vanish because of the dimension constrains. Indeed,  $\dim \overline{\mathcal{M}}_{g,n}=3g-3+n$ and $\Theta_{g,n}\in H^{4g-4+2n}(\overline{\mathcal{M}}_{g,n})$, therefore $\Theta_{g,n}^2=0$. 
\end{remark}
Again, the integral on the right hand side vanishes unless the stability condition~\eqref{stability} is satisfied, all $a_i$ are non-negative integers, and the dimension constraint 
\be\label{d2}
g-1=\sum_{i=1}^n a_i
\ee 
holds. Consider the generating function of the 
intersection numbers of $\Theta$ and $\psi$ classes
\be
F^\Theta_{g,n}= \sum_{a_1,\ldots,a_n\ge 0} \<\tau_{a_1} \tau_{a_2} \cdots \tau_{a_n}\>_g^\Theta \frac{\prod T_{a_i}}{n!} 
\ee
then, we have a direct analog of the Kontsevich--Witten tau-function:
\begin{theorem*}[\cite{Norb}]
The generating function
\be
\tau_\Theta = \exp\left(\sum_{g=0}^\infty \sum_{n=0}^\infty \hbar^{2g-2+n}F_{g,n}^\Theta \right)
\ee
becomes a tau-function of the KdV hierarchy after the change of variables~$T_k=(2k+1)!!t_{2k+1}$.
\end{theorem*}
Norbury also proved, that $\tau_\Theta$ is nothing but a tau-function of the BGW model, 
\be
\tau_\Theta=\tau_{BGW}.
\ee
KdV integrability of the BGW model follows from the relation to the generalized Kontsevich
model and was established by Mironov, Morozov and Semenoff \cite{MMS}.


\subsection{Heisenberg--Virasoro constraints}\label{SHV}

It appears that the KW and BGW tau-functions share a lot of common properties and it is natural to consider them simultaneously.
Let us denote
\be
\tau_1={\tau}_{KW},\quad \quad \tau_0={\tau}_{BGW}.
\ee
The dimension constrains \eqref{d1}, \eqref{d2} can be represented as
\begin{equation}
\begin{split}\label{dimV}
 (1+2\alpha) \hbar \frac{\p}{\p \hbar}  {\tau}_\alpha=\widehat{L}_0 \cdot {\tau}_\alpha , \,\,\,\,\,\,\,\,\,\alpha=0,1.
\end{split}
\end{equation}
Below we always assume that the index $\alpha$ runs from $0$ to $1$,
\be
\alpha\in \{0,1\}.
\ee
\begin{remark}
Tau-functions $ {\tau}_\alpha$ are KdV tau-functions; however, it is convenient to consider them in the context of the KP hierarchy. So, below we  use the 
Heisenberg--Virasoro algebra of KP symmetries, introduced in Section \ref{S1}.
\end{remark}

Both $\tau_\alpha$ are KdV tau-functions, hence satisfy the {\em Heisenberg constraints} (or the KdV reduction constraints)
\be\label{Heis}
\frac{\p}{\p t_{2k}} {\tau}_\alpha=0, \,\,\,\,\,\,\,\,\,\,\,\, k>0.
\ee
Moreover, they satisfy the {\em Virasoro constraints}  \cite{Fuku,DVV,GN} 
\begin{equation}
\begin{split}\label{VirC}
\widehat{L}_k^{\alpha} \cdot {\tau}_{\alpha} &=0, \,\,\,\, k \geq -\alpha, \\
\end{split}
\end{equation}
where the Virasoro operators are given by
\begin{equation}
\begin{split}
\widehat{L}_k^{\alpha}&=\frac{1}{2}\widehat{L}_{2k}-\frac{1}{2\hbar}\frac{\p}{\p t_{2k+1+2\alpha}} +\frac{\delta_{k,0}}{16} \in {\mathcal L}.
\end{split}
\end{equation}
These operators satisfy the commutation relations
\be
\left[\widehat{L}_k^{\alpha},\widehat{L}_k^{\alpha}\right]=(k-m)\widehat{L}_{k+m}^{\alpha}, \quad \quad k,m \geq -\alpha.
\ee
Combining (\ref{Heis}) with (\ref{VirC}) one gets
\be
\widehat{L}_0 \cdot \tau_\alpha =(1+2\alpha )\hbar \widehat{W}_\alpha \cdot \tau_\alpha,
\ee
where
\be\label{leadcaj}
\widehat{W}_\alpha=\frac{1}{2\alpha+1} \sum_{k=0}^\infty (2k+1) t_{2k+1}\left(\widehat{L}_{2k-2\alpha}+\frac{\delta_{k,\alpha}}{8}\right).
\ee
Therefore, from (\ref{dimV}) we have
\be
\hbar \frac{\p}{\p \hbar} \tau_\alpha=\hbar \widehat{W}_\alpha \cdot \tau_\alpha.
\ee

This is the relation of the form (\ref{gene}), it leads to the cut-and-join description for the KW and BGW tau-functions \cite{KSS,KS2}
\be\label{cajkdv}
\tau_{\alpha}=\exp\left({\hbar \widehat{W}_{\alpha}}\right)\cdot 1.
\ee
\begin{remark}\label{remp}
There is a certain arbitrariness in the construction of the cut-and-join operators.
Since the tau-functions $\tau_\alpha$ depend only on odd times, it is possible to consider operators $\widehat{W}_\alpha$ with omitted even times $t_{2k}$ and corresponding derivatives,
\begin{equation}\label{CAJ}
\begin{split}
\widehat{W}_0^*&=\sum_{k,m\in {\mathbb Z}_{\odd}^+}^\infty \left(kmt_{k}t_{m}\frac{\p}{\p t_{k+m-1}}+\frac{1}{2}(k+m+1)t_{k+m+1}\frac{\p^2}{\p t_k \p t_m}\right)+\frac{t_1}{8},\\
\widehat{W}_1^*&=\frac{1}{3}\sum_{k,m\in {\mathbb Z}_{\odd}^+}^\infty \left(kmt_{k}t_{m}\frac{\p}{\p t_{k+m-3}}+\frac{1}{2}(k+m+3)t_{k+m+3}\frac{\p^2}{\p t_k \p t_m}\right)+\frac{t_1^3}{3!}+\frac{t_3}{8}.
\end{split}
\end{equation}
Then an equivalent cut-and-join description is given by the these operators,
\be
\tau_{\alpha}=\exp\left({\hbar \widehat{W}_{\alpha}^*}\right)\cdot 1.
\ee
\end{remark}

If we consider the topological expansion of the tau-functions $\tau_\alpha=\sum_{k=0}^\infty \tau_\alpha^{(k)}\hbar^k$, then we have the algebraic topological recursion
\begin{equation}
\begin{split}
\tau_\alpha^{(k)}&=\frac{1}{k} \widehat{W}_\alpha\cdot \tau_\alpha^{(k-1)}\\
&=\frac{1}{k} \widehat{W}_\alpha^*\cdot \tau_\alpha^{(k-1)}
\end{split}
\end{equation}
with the initial condition $\tau_\alpha^{(0)}=1$.
Moreover, we can separate the contributions with different numbers of marked points. Namely, let $\tau_\alpha^{(k)}=\sum_{n=1}^{(2\alpha+1)(k-1)} \tau_{\alpha,n}^{(k)}$, where $\tau_{\alpha,n}^{(k)}$ is a homogeneous polynomial in ${\bf t}$ of degree $n$. Let also
\begin{equation}
\begin{split}
\widehat{W}_{0,1}^*&=\sum_{k,m\in {\mathbb Z}_{\odd}^+}^\infty kmt_{k}t_{m}\frac{\p}{\p t_{k+m-1}}+\frac{t_1}{8},\\
\widehat{W}_{0,-1}^*&=\frac{1}{2}\sum_{k,m\in {\mathbb Z}_{\odd}^+}^\infty (k+m+1)t_{k+m+1}\frac{\p^2}{\p t_k \p t_m},\\
\widehat{W}_{1,1}^*&=\frac{1}{3}\sum_{k,m\in {\mathbb Z}_{\odd}^+}^\infty kmt_{k}t_{m}\frac{\p}{\p t_{k+m-3}}+\frac{t_3}{8},\\
\widehat{W}_{1,-1}^*&=\frac{1}{6}\sum_{k,m\in {\mathbb Z}_{\odd}^+}^\infty (k+m+3)t_{k+m+3}\frac{\p^2}{\p t_k \p t_m},
\end{split}
\end{equation}
and $\widehat{W}_{0,3}^*=0$, $\widehat{W}_{1,3}^*=\frac{t_1^3}{3!}$. Then we have the recursion
\be\label{fixn}
\tau_{\alpha,n}^{(k)}=\frac{1}{k}\left( \widehat{W}^*_{\alpha,3}\cdot\tau_{\alpha,n-3}^{(k-1)}+\widehat{W}^*_{\alpha,1}\cdot\tau_{\alpha,n-1}^{(k-1)}+\widehat{W}^*_{\alpha,-1}\cdot\tau_{\alpha,n+1}^{(k-1)}\right)
\ee
with the initial condition $\tau_{\alpha,n}^{(0)}=\delta_{n,0}$.

The cut-and-join description allows us to find the topological expansion of the tau-functions
\begin{multline}
\log \tau_0={\frac {t_{{1}}}{8}}\hbar +{\frac {{t_{{1}}}^{2}}{16}}{\hbar }^{2}+
 \left( {\frac {{t_{{1}}}^{3}}{24}}+{\frac {9\,t_{{3}}}{128}} \right) 
{\hbar }^{3}+ \left( {\frac {{t_{{1}}}^{4}}{32}}+{\frac {27\,t_{{3}}t_
{{1}}}{128}} \right) {\hbar }^{4}+ \left( {\frac {{t_{{1}}}^{5}}{40}}+
{\frac {27\,t_{{3}}{t_{{1}}}^{2}}{64}}+{\frac {225\,t_{{5}}}{1024}}
 \right) {\hbar }^{5}\\
 + \left( {\frac {{t_{{1}}}^{6}}{48}}+{\frac {45\,
t_{{3}}{t_{{1}}}^{3}}{64}}+{\frac {1125\,t_{{5}}t_{{1}}}{1024}}+{
\frac {567\,{t_{{3}}}^{2}}{1024}} \right) {\hbar }^{6}+O(\hbar^7),
\end{multline}
\begin{multline}
\log \tau_1=\left( {\frac {{t_{{1}}}^{3}}{6}}+{\frac {t_{{3}}}{8}} \right) 
\hbar + \left( {\frac {t_{{3}}{t_{{1}}}^{3}}{2}}+{\frac {5\,t_{{5}}t_{
{1}}}{8}}+{\frac {3\,{t_{{3}}}^{2}}{16}} \right) {\hbar }^{2}\\
+ \left( 
{\frac {15\,t_{{1}}t_{{3}}t_{{5}}}{4}}+{\frac {3\,{t_{{3}}}^{2}{t_{{1}
}}^{3}}{2}}+{\frac {5\,t_{{5}}{t_{{1}}}^{4}}{8}}+{\frac {35\,t_{{7}}{t
_{{1}}}^{2}}{16}}+{\frac {3\,{t_{{3}}}^{3}}{8}}+{\frac {105\,t_{{9}}}{
128}} \right) {\hbar }^{3}+O(\hbar^4).
\end{multline}


\subsection{Heisenberg--Virasoro group action and cut-and-join operators}\label{S3.3}
In this section, we investigate how the subgroup ${\mathcal V}_+$ of the Heisenberg--Virasoro group, see Section \ref{S1}, acts on the KdV tau-functions $\tau_\alpha$, $\alpha \in \{0,1\}$. In general, the action of this subgroup is not well-defined and can lead to divergences. In this section, we do not address this potential issue, assuming that all considered functions are well defined.

Let us consider the subgroups of the Heisenberg--Virasoro group, generated by the operators of the form 
\be\label{opG}
{\mathcal V}_+\ni \widehat{G}_\alpha:=\exp\left(\sum_{k=1}^\infty \tilde{a}_k  \widehat{L}_k\right) \exp\left(\hbar^{-1} \sum_{k=2+2\alpha}^\infty {v}_k \widehat{J}_k \right).
\ee
\begin{remark}
One can easily relax the constraint on the range of summation and consider more general operators. However, this makes the formulas less compact, so we prefer to deal with a less general case, which is sufficient for our needs.
\end{remark}

We consider the functions $\widehat{G}_\alpha \cdot \tau_\alpha$.
The tau-functions $\tau_\alpha$ satisfy the KdV reduction constraints, $\widehat{J}_{2k} \cdot \tau_\alpha=0$ for $k>0$, therefore we may assume that ${v}_{2k}=0$. 
Using the Virasoro constrains (\ref{VirC}) we can rewrite action of the operator \eqref{opG} on the tau-function $\tau_\alpha$ in terms of the Virasoro group only. For this purpose we introduce a bijection between the space of the Virasoro operators 
\be\label{Virg}
\widehat{V}_{\bf a}=\exp\left({\sum_{k>0} a_k \widehat{L}_k }\right)
\ee
 and the space of all formal series of the form $f(z)\in z+z{\mathbb C}[\![z]\!]$. Recall that the operators ${\mathtt l}_k$ are given by \eqref{virw}. 
\begin{definition}\label{def1}
For the Virasoro group element $\widehat{V}$ consider
\be
\Xi(\widehat{V})=e^{\sum_{k>0} a_k {\mathtt l}_k } \,z \, e^{-\sum_{k>0} a_k {\mathtt l}_k } \in z+z{\mathbb C}[\![z]\!].
\ee
For a given set of parameters $a_k\in {\mathbb C}$, $k\in {\mathbb Z}_{>0}$ we denote such a  series  associated to the operator ${\widehat{V}}_{\bf a}$ by $f(z;{\bf a})=\Xi(\widehat{V}_{\bf a})$.
We also introduce the inverse formal series $h(z;{\bf a})$,
\be
f(h(z;{\bf a});{\bf a})=z.
\ee
\end{definition}
For any series $f(z)\in z+z{\mathbb C}[\![z]\!]$ we construct a corresponding element of the Virasoro group 
\be
\widehat{V}=\Xi^{-1}(f(z)),
\ee
where $a_k$ are defined implicitly by
\be
f(z)=e^{\sum_{k>0} a_k {\mathtt l}_k } \,z \, e^{-\sum_{k>0} a_k {\mathtt l}_k } .
\ee
For the inverse group element we have
\be
\widehat{V}^{-1}=\Xi^{-1}(h(z)).
\ee
Multiplication of the Virasoro group elements corresponds to the composition of the associated functions
\be
\Xi(\widehat{V}_1\widehat{V}_2)=f_2(f_1(z)).
\ee

To any sequence of coefficients $a_k\in {\mathbb C}$, $k\in {\mathbb Z}_{>0}$, we also associate the following formal series
\be\label{vfunc}
{v}^\alpha(z)=\frac{1}{(2\alpha+1)}  (z^{2\alpha+1}-f(z;{\bf {a}})^{2\alpha+1}) \in z^{2+2\alpha}{\mathbb C}[\![z]\!]
\ee
with the coefficients
\be
v_k^{\alpha}({\bf{a}})= [z^k] {v}^\alpha(z),
\ee
where for any formal Laurent series $F(z)=\sum b_k z^k$ we introduce $[z^k]F(z):=b_k$.
\begin{lemma}\label{virtos}
We have
\be
\exp\left({\hbar^{-1} \sum_{k=\alpha+1} v_{2k+1}^{\alpha}({\bf{a}}) \frac{\p}{\p t_{2k+1}}}\right) \cdot \tau_\alpha=\exp\left({\sum_{k>0} a_{2k} \widehat{L}_{2k} }\right) \cdot \tau_\alpha.
\ee
\end{lemma}
\begin{proof}
The statement of the lemma follows from Lemma 1.1 of \cite{A1}.
\end{proof}
From this lemma it follows that
\begin{equation}\label{VtoJ}
\begin{split}
\widehat{G}_\alpha \cdot \tau_\alpha 
&=\exp\left(\sum_{k=1}^\infty \tilde{a}_k  \widehat{L}_k\right) \exp\left(\hbar^{-1} \sum_{k=\alpha+1}^\infty {v}_{2k+1} \widehat{J}_{2k+1} \right) \cdot \tau_\alpha \\
&= \exp\left(\sum_{k=1}^\infty \tilde{a}_k  \widehat{L}_k\right)  \exp\left(\sum_{k=1}^\infty {a}_{2k}  \widehat{L}_{2k}\right) \cdot \tau_\alpha,
\end{split}
\end{equation}
where the coefficients ${a}_k$ are given implicitly by the relation \eqref{vfunc} with ${v}_{2k+1}=v_{2k+1}^{\alpha}({\bf{a}})$. Therefore, we always can reduce the action of the operator \eqref{opG} to the action of equivalent operator \eqref{Virg} from the Virasoro group.

Let us consider arbitrary operator $\widehat{V}_{\bf a}$ of the form \eqref{Virg} and its action on the tau-function
\be\label{relV} 
\tau_\alpha^{\bf a}=\widehat{V}_{\bf a} \cdot \tau_\alpha.
\ee
for some coefficients $a_k$. $\tau_\alpha^{\bf a}$ are tau-functions of the KP hierarchy. Let us derive the Heisenberg--Virasoro constraints and the cut-and-join description for the tau-functions $\tau_\alpha^{\bf a}$.

We introduce
\begin{equation}\label{LJa}
\begin{split}
{\widehat{J}}_{k}^{{\bf a}}=\widehat{V}_{\bf a}  \widehat{J}_{k} \widehat{V}_{\bf a}^{-1}, \quad \quad
{\widehat{L}}_{k}^{{\bf a}}=\widehat{V}_{\bf a}   \widehat{L}_{k} \widehat{V}_{\bf a}^{-1}.
\end{split}
\end{equation}
From the commutation relations of the Heisenberg--Virasoro algebra ${\mathcal L}$ it follows that
these operators are given by the linear combinations of $\widehat{J}_{m}$ and $ \widehat{L}_{m}$ operators.
Let us also introduce the coefficients
\be\label{sigmar}
\rho[k,m]= [z^m] f(z;{\bf{a}})^k,\quad \quad  \quad \sigma[k,m]=  [z^{m+1}]\frac{f(z;{\bf{a}})^{k+1}}{f'(z;{\bf{a}})},
\ee
where $f(z,{\bf a})=\Xi(\widehat{V}_{\bf a})$. They are polynomials in ${\bf a}$.

From the relation between the operators $ \widehat{J}_{k}$, $ \widehat{L}_{k}$, ${\mathtt l}_k$, ${\mathtt j}_k$ and Definition \ref{def1} we have
\begin{lemma}\label{lemma32}
\begin{equation}\label{rrr}
 \begin{aligned}
 \widehat{J}_{k}^{{\bf a}}&= \sum_{m=k}^\infty \rho[k,m] \widehat{J}_{m},                                                     & &k\in {\mathbb Z},\\
    \widehat{L}_{k}^{{\bf a}}&= \sum_{m=k}^\infty \sigma[k,m] \widehat{L}_{m}+\frac{a_2}{2}\delta_{k, -2},   \quad \quad & &k\geq -2.
\end{aligned}
\end{equation}
\end{lemma}

Consider the operators
\begin{equation}
\begin{split}
\widehat{J}_k^{{\bf a}}&=\frac{1}{2} \widehat{J}_{2k}^{{\bf a}},\\
\widehat{L}_k^{\alpha,{\bf a}}&=\frac{1}{2}\widehat{L}_{2k}^{\bf a}-\frac{1}{2\hbar} \widehat{J}_{2k+1+2\alpha}^{{\bf a}} +\frac{\delta_{k,0}}{16}.
\end{split}
\end{equation}
Then
\begin{lemma}\label{lemmavir}
Tau-functions $\tau_\alpha^{\bf a}$ satisfy the Heisenberg constraints
\be
 \widehat{J}_{2k}^{{\bf a}}\cdot \tau_\alpha^{\bf a}=0, \quad k>0,
\ee
and the Virasoro constraints
\begin{equation}
\begin{split}\label{VirCn}
\widehat{L}_k^{\alpha,{\bf a}} \cdot \tau_\alpha^{\bf a} &=0, \quad k \geq -\alpha.
\end{split}
\end{equation}
These operators satisfy the commutation relations of the Heisenberg--Virasoro algebra
\begin{equation}
 \begin{aligned}
\left[\widehat{J}_k^{{\bf a}},\widehat{J}_m^{{\bf a}}\right]&=0,                                                                     & &k,m>0,\\
\left[\widehat{L}_k^{\alpha,{\bf a}},\widehat{J}_m^{{\bf a}}\right]&=k  \widehat{J}_{k+m}^{{\bf a}},                             & &k\geq -\alpha,\, m>0,\\
\left[\widehat{L}_k^{\alpha,{\bf a}},\widehat{L}_m^{\alpha,{\bf a}}\right]&=(k-m)\widehat{L}_{k+m}^{\alpha,{\bf a}}, \quad  \quad   & &k,m\geq -\alpha.
\end{aligned}
\end{equation}
\end{lemma}
\begin{proof}
Operators $\widehat{J}_k^{{\bf a}}$ and $\widehat{L}_k^{\alpha,{\bf a}}$
are given by the conjugation of the constraints \eqref{Heis} and \eqref{VirC}, in particular
\be
\widehat{L}_k^{\alpha,{\bf a}}= \widehat{V}_{\bf a}   \widehat{L}_k^{\alpha} \widehat{V}_{\bf a}^{-1}.
\ee
\end{proof}

If ${\bf a}$ do not depend on $\hbar$, then
the relation \eqref{relV} is of the form \eqref{aneq}, therefore we can derive the cut-and-join description of $\tau_\alpha^{\bf a}$ from the expression \eqref{leadcaj}.
Let
\be\label{newcaj}
\widehat{{W}}_{\alpha}^{{\bf a}}=\frac{1}{2\alpha+1} \sum_{k=0}^\infty  \widehat{J}_{-2k-1}^{{\bf a}}\left(\widehat{L}_{2k-2\alpha}^{{\bf a}}+\frac{\delta_{k,\alpha}}{8}\right).
\ee
This is a cut-and-join operator for $\tau_\alpha^{\bf a}$.
\begin{lemma}\label{lemma3.3}
\be
\tau_\alpha^{\bf a}= \exp\left({\hbar \widehat{{W}}_{\alpha}^{{\bf a}}}\right) \cdot 1.
\ee
\end{lemma}
\begin{proof}
We use the cut-and-join representation (\ref{cajkdv}). Since $\widehat{L}_k \cdot 1=0$ for $k\geq 0$,  
we have
\begin{equation}
\begin{split}
\widehat{V}_{\bf a} \cdot \tau_\alpha  &=  \widehat{V}_{\bf a} \exp\left({\hbar \widehat{W}_{\alpha}}\right) \widehat{V}_{\bf a}^{-1} \cdot 1\\
&=\exp\left({\hbar \widehat{{W}}_{\alpha}^{{\bf a}}}\right) \cdot 1,
\end{split}
\end{equation}
where $\widehat{{W}}_{\alpha}^{{\bf a}}:=\widehat{V}_{\bf a}   \widehat{W}_{\alpha}\widehat{V}_{\bf a}^{-1}$
coincides with \eqref{newcaj}.
\end{proof}
Because of the certain arbitrariness of the cut-and-join description of the tau-functions $\tau_\alpha$, see Remark \ref{remp}, we also construct another version of the cut-and-join operators for the tau-functions $\tau_\alpha^{\bf a}$. Let
\begin{equation}
\begin{split}\label{Wpri}
{\widehat{W}_0^{*{\bf a}}}&=\sum_{k,m\in {\mathbb Z}_{\odd}^+}^\infty\left( {\widehat {J}}_{-k}^{\bf a}  {\widehat {J}}_{-m}^{\bf a}  {\widehat {J}}_{k+m-1}^{\bf a} +\frac{1}{2} {\widehat {J}}^{\bf a}_{-k-m-1} {\widehat {J}}_{k}^{\bf a}  {\widehat {J}}_{m}^{\bf a}\right)+\frac{{\widehat {J}}_{-1}^{\bf a}}{8},\\
{\widehat{W}_1^{*{\bf a}}}&=\frac{1}{3}\sum_{k,m\in {\mathbb Z}_{\odd}^+}^\infty \left(  {\widehat {J}}_{-k}^{\bf a}  {\widehat {J}}_{-m}^{\bf a}   {\widehat {J}}_{k+m-3}^{\bf a}
+\frac{1}{2} {\widehat {J}}_{-k-m-3}^{\bf a} {\widehat {J}}_{k}^{\bf a}  {\widehat {J}}_{m}^{\bf a}\right)+\frac{ ({\widehat {J}}_{-1}^{\bf a})^3}{3!}+\frac{ {\widehat {J}}_{-3}^{\bf a}}{24}.
\end{split}
\end{equation}
\begin{lemma}\label{l35}
\be
\tau_\alpha^{\bf a}=\exp\left(\hbar\, {\widehat{W}_\alpha^{*{\bf a}}}\right)\cdot 1.
\ee
\end{lemma}
Note that the operators ${\widehat{W}_\alpha^{\bf a}}$ and ${\widehat{W}_\alpha^{*{\bf a}}}$ only contain the cubic and linear  terms in $ \widehat{J}_{k}$.


\section{Triple Hodge integrals and KP hierarchy}\label{H3}

In \cite{H3_1,H3_2} the author has proven that the generating functions of the triple Hodge integrals satisfying the Calabi--Yau condition are tau-functions of the KP hierarchy. These tau-functions can be related to the tau-functions $\tau_\alpha$ considered in Section \ref{S3} by certain elements of the subgroup ${\mathcal V}_+$ of the Heisenberg--Virasoro group. The Virasoro constraints allow us to reduce these group elements to the elements of the Virasoro group. In this section, using the arguments of Section \ref{S3.3} we will derive the cut-and-join description and the Heisenberg--Virasoro constraints for the generating functions of the triple Hodge integrals satisfying the Calabi--Yau condition.


\subsection{KP integrability}\label{S_KP}

Let us remind the description of  the generating functions of the triple Hodge integrals satisfying the Calabi--Yau condition by KP tau-functions \cite{H3_1,H3_2}.
Consider the linear change of variables
\begin{equation}\label{changeof}
\begin{split}
T_0^{q,p}({\bf t})&=t_1,\\
T_k^{q,p}({\bf t})&=\sum_{m=1}^\infty m t_m \left(q \frac{\p}{\p t_m}+\frac{2q+p}{\sqrt{p+q}}\frac{\p}{\p t_{m-1}} +\frac{\p}{\p t_{m-2}}\right)\cdot T_{k-1}^{q,p}({\bf t}),\quad k>1,
\end{split}
\end{equation}
where $q$ and $p$ are two parameters with $q+p\neq 0$ if $q\neq 0$.

We also introduce the coefficients $a_k^{q,p}$. They are given implicitly through the Definition \ref{def1},
\be\label{ff}
f(z;{\bf a}^{q,p})=f(z),
\ee
where the series on the right hand side is given by
\be\label{xdd}
f(z) \hbox{d} f(z)=\frac{z \hbox{d}z}{(1+\sqrt{p+q}z)(1+q z/\sqrt{p+q})}.
\ee 
In particular,
\be\label{aa}
a_1^{q,p}=\frac{1}{3}{\frac {p+2\,q}{\sqrt {p+q}}}, \quad a_2^{q,p}=-\frac{1}{12}{\frac {{p}^{2}+pq+{q}^{2}}{p+q}}, \quad a_3^{q,p}={\frac {31\,{p}^{3}+69\,{p}^{2}q+21\,p{q}^{2}+14\,{q}^{3}}{1080\, \left( p+q \right) ^{3/2}}}.
\ee

Let ${\mathtt x}(z)=\frac{f^2(z)}{2}=z^2/2+\dots$.
If $p\neq 0$ and $q\neq 0$ we have
\be\label{fpq}
{\mathtt x}(z)=\frac{p+q}{pq}\log\left(1+\frac{qz}{\sqrt{p+q}}\right)-\frac{1}{p}\log(1+\sqrt{p+q}z)
\ee
and
\be\label{fpq1}
f(z)=\sqrt{2\frac{p+q}{pq}\log\left(1+\frac{qz}{\sqrt{p+q}}\right)-\frac{2}{p}\log(1+\sqrt{p+q}z)}.
\ee
For $q=0$ it degenerates to
\be\label{xq}
{\mathtt x}(z)=\frac{z}{\sqrt{p}}-\frac{1}{p}\log(1+\sqrt{p}z),
\ee
and for $p=0$ it degenerates to
\be
{\mathtt x}(z)=\frac{1}{q}\log(1+\sqrt{q}z)-\frac{1}{\sqrt{q}}\frac{z}{1+\sqrt{q} z}.
\ee

Let us note that the change of variables \eqref{changeof} is generated by the operator
\be
\frac{\p}{\p {\mathtt x}}=\frac{(1+\sqrt{p+q}z)(1+q z/\sqrt{p+q})}{z} \frac{\p}{\p z}.
\ee
Namely, the functions 
\be\label{phik}
\Phi_k(z):=\left(-\frac{\p}{\p {\mathtt x}}\right)^k\cdot \frac{1}{z}
\ee
define this change of variables if we associate $k t_k$ with $z^{-k}$ and $T_k^{q,p}$ with $\Phi_k(z)$. 

Let
\be{\mathtt y}_0(z)=\frac{1}{z}
\ee
and 
\be\label{ygr}
{\mathtt y}_1(z)=\int_0^z{\mathtt y}_0(\eta) \hbox{d} {\mathtt x}(\eta) =\int_0^z \frac{\hbox{d} {\mathtt x}(\eta)}{\eta},
\ee
where $ \hbox{d} {\mathtt x}(\eta)=f(\eta)  \hbox{d} f(\eta)$. For $p\neq 0$ and $q\neq 0$ we have
\be\label{y1}
{\mathtt y}_1(z) =\frac {\sqrt {p+q}}{p}   \left( \log  \left( 1+\sqrt {p+q}z \right)  -\log  \left(1+\frac{qz}{ \sqrt {p+q}} \right) \right).
\ee
For $q=0$ it reduces to
\be
{\mathtt y}_1(z)=\frac {1}{\sqrt{p}} \log  \left( 1+\sqrt {p }z \right)
\ee
and for $p=0$ it degenerates to
\be
{\mathtt y}_1(z)=\frac{z}{1+\sqrt{q}z}.
\ee
The series ${\mathtt x}$ and ${\mathtt y}_1$ satisfy the equation
 \be\label{spp}
p\exp\left(-q{\mathtt x}(z)\right)=(p+q)\exp\left(\frac{q}{\sqrt{p+q}}{\mathtt y}_1(z)\right)-q\exp\left(\sqrt{p+q}{\mathtt y}_1(z)\right).
\ee

 Let us consider the Hodge bundle $\mathbb E$, a rank $g$ vector bundle over $\overline{\mathcal{M}}_{g,n}$. For $\lambda_i=c_i({\mathbb E})$ we set
\be
\Lambda_g(u)=\sum_{i=0}^g u^i \lambda_i.
\ee
Below we focus on the case of cubic Hodge integrals
\be
\int_{\overline{\mathcal{M}}_{g,n}} \Theta_{g,n}^{1-\alpha} \Lambda_g (u_1) \Lambda_g (u_2) \Lambda_g (u_3)\psi_1^{a_1} \psi_2^{a_2} \cdots \psi_n^{a_n}
\ee
satisfying an additional {\em Calabi--Yau condition}
\be\label{specialcond}
\frac{1}{u_1}+\frac{1}{u_2}+\frac{1}{u_3}=0.
\ee
For this case it is convenient to use the parametrization
\be\label{speccond}
u_1=-p, u_2=-q, u_3=\frac{pq}{p+q}.
\ee
We introduce the generating functions of the triple Hodge integrals satisfying the Calabi--Yau condition, with (for $\alpha=0$) and without (for $\alpha=1$) Norbury's $\Theta$-classes, namely
\be\label{GenfN}
Z_{q,p}^{(\alpha)}({\bf T})=\exp\left(\sum_{g=0}^\infty \sum_{n=0}^\infty \hbar^{2g-2+n} {\mathcal F}_{g,n}^{\alpha}\right),
\ee
where
\be\label{FF}
{\mathcal F}_{g,n}^\alpha=\sum_{a_1,\dots,a_n\ge 0}\frac{\prod T_{a_i}}{n!}\int_{\overline{\mathcal{M}}_{g,n}} \Theta_{g,n}^{1-\alpha} \Lambda_g (-q) \Lambda_g (-p) \Lambda_g \left(\frac{pq}{p+q}\right)\psi_1^{a_1} \psi_2^{a_2} \cdots \psi_n^{a_n}.
\ee
Define the coefficients ${v}_j^{(\alpha)}$ by
\begin{equation}\label{vdef}
\begin{split}
\sum_{j=2+2\alpha}^\infty {v}_j^{(\alpha)} z^j &=\int_0^z \left(f(\eta)^{2\alpha-1}- {\mathtt y}_\alpha(\eta)\right) \hbox{d}{\mathtt x}(\eta).
\end{split}
\end{equation}
In particular, for $\alpha=0$ one has
\begin{equation}\label{vdef1}
\begin{split}
v^{(0)}_k&=[z^k]\left(f(z)- {\mathtt y}_1(z)\right).
\end{split}
\end{equation}
Let us introduce an element of the Virasoro group
\be
\widehat{V}_{q,p}=\Xi^{-1}(f(z))=\exp\left({\sum_{k\in{\mathbb Z}_{>0}} a_k^{q,p} \widehat{L}_k}\right).
\ee

Then we have
\begin{theorem*}[\cite{H3_1}]
The generating functions of triple Hodge integrals and $\Theta$-Hodge integrals satisfying the Calabi--Yau condition in the variables ${{\bf T}^{q,p}(\bf t)}$ are tau-functions of the KP hierarchy,
\be\label{tau^alpha}
\tau_{q,p}^{(\alpha)}({\bf t})=Z_{q,p}^{(\alpha)}({{\bf T}^{q,p}(\bf t)}).
\ee
They are related to the KW and BGW tau-functions by the elements of the Heisenberg--Virasoro group 
\be\label{qptau}
\tau_{q,p}^{(\alpha)}({\bf t})=\exp\left({\hbar^{-1} \sum_{j=2+2\alpha}^\infty v_j^{(\alpha)} \frac{\p}{\p t_{j}}}\right)
\widehat{V}_{q,p} \cdot  \tau_{\alpha}({\bf t}).
\ee
\end{theorem*}
The proof of the theorem is based on the comparison between the Givental operators and the elements of the  Heisenberg--Virasoro group ${\mathcal V}$ of KP symmetries.
For the $p=0$ and $\alpha=1$ case, corresponding to the linear Hodge integrals, the KP integrability  was proved by Kazarian \cite{Kaza}. Relation between tau-functions given by the elements of the Heisenberg--Virasoro group elements for this case was derived in \cite{A1}. Identification of this group element and the Givental group element has been found in \cite{Wang}. Alternative proof of the integrability part of this theorem for $\alpha=1$, based on the Mari\~{n}o--Vafa formula, is given by Kramer \cite{Reinier}. 

\begin{remark}
In this paper we focus on the tau-functions $\tau_{q,p}^{(\alpha)}$ and do not investigate the properties of the generating functions $Z_{q,p}^{(\alpha)}$ in the original variables. Virasoro constraints and cut-and-join description for them can be found with the help of the Givental formalism \cite{Giv1,Teleman} and will be considered elsewhere. 
\end{remark}

\begin{remark}
Some other relations between triple Hodge integrals satisfying the Calabi--Yau condition and integrable systems are known, see \cite{H3_1} for the references. 
\end{remark}
\subsection{Heisenberg--Virasoro and dimension constraints}\label{S42}

We start with the permutation of the translation operator with the element of the Virasoro group in \eqref{qptau}. 
Namely, Lemma \ref{lemma32} leads to the relation
\be\label{le}
\tau_{q,p}^{(\alpha)}({\bf t})=\widehat{V}_{q,p} \exp \left({\hbar^{-1} \sum_{j=2+2\alpha} \tilde{v}_j^{(\alpha)} \frac{\p}{\p t_{j}}}\right)\cdot  \tau_{\alpha}({\bf t}),
\ee
where from \eqref{vdef} for the coefficients $\tilde{v}_j^{(\alpha)} $ we have
\begin{equation}\label{vtilde}
\begin{split}
\sum_{j=2+2\alpha}^\infty \tilde{v}_j^{(\alpha)} z^j =\int_0^{h(z)} \left(f(\eta)^{2\alpha-1}-{\mathtt y}_\alpha(\eta)\right) \hbox{d}{\mathtt x}(\eta).
\end{split}
\end{equation}
Here $h(z)$ is a series, inverse to $f(z)$ of \eqref{fpq}, that is $h(f(z))=f(h(z))=z$, hence
\be\label{vtilde1}
\sum_{j=2+2\alpha}^\infty \tilde{v}_j^{(\alpha)} z^j =\int_0^z \left(\eta^{2\alpha}-\eta {\mathtt y}_\alpha(h(\eta))\right) \hbox{d}\eta.
\ee
Since $\tau_{\alpha}({\bf t})$ are tau-function of the KdV hierarchy, we can neglect the translations of the even times $t_{2k}$ and
\be\label{transn}
\tau_{q,p}^{(\alpha)}({\bf t})=\widehat{V}_{q,p}  \exp\left({\hbar^{-1} \sum_{j=1+\alpha} \tilde{v}_{2j+1}^{(\alpha)} \frac{\p}{\p t_{2j+1}}}\right)\cdot  \tau_{\alpha}({\bf t}).
\ee
To describe the Heisenberg--Virasoro constraints for the tau-functions $\tau_{q,p}^{(\alpha)}$ 
let us consider
\begin{equation}
 \begin{aligned}
 \widehat{J}_k^{(q,p)} &=\frac{1}{2}\sum_{m=2k}^\infty \rho[2k,m] \widehat{J}_{m},\\ 
\widehat{L}_k^{\alpha,(q,p)}&=\frac{1}{2}\sum_{m=2k}^\infty \sigma[2k,m] \widehat{L}_{m}- \frac{1}{2\hbar}\sum_{m=2k+1+2\alpha}^\infty\chi_\alpha[2k,m] \widehat{J}_{m}
 +\frac{\delta_{k,0}}{16}-\frac{\delta_{k,-1}}{48}{\frac {{p}^{2}+pq+{q}^{2}}{p+q}}.
\end{aligned}
\end{equation}
Here the coefficients $\sigma[k,m]$, $\rho[k,m]$, and $\chi_\alpha[k,m]$ are given by
\be
\rho[k,m]= [z^m] f(z)^k,\quad \quad \sigma[k,m]= [z^{m+1}]\frac{f(z)^{k+1}}{f'(z)},\quad \quad \chi_\alpha[k,m]= [z^m] f(z)^{k+2}{\mathtt y}_\alpha(z),
\ee
where the function $f(z)$ is given by \eqref{fpq1}. 
\begin{theorem}\label{propp}
The tau-functions of triple Hodge integrals satisfy the Heisenberg--Virasoro constraints
\begin{equation}
 \begin{aligned}
\widehat{J}_k^{(q,p)} \cdot {\tau}_{q,p}^{(\alpha)}&=0, \quad & &k>0,\\
\widehat{L}_k^{\alpha,(q,p)} \cdot {\tau}_{q,p}^{(\alpha)} &=0,\quad \quad & &k \geq -\alpha.
\end{aligned}
\end{equation}
These operators satisfy the commutation relations of the Heisenberg--Virasoro algebra
\begin{equation}
 \begin{aligned}
\left[\widehat{J}_k^{(q,p)},\widehat{J}_m^{(q,p)}\right]&=0,                                                                     & &k,m>0,\\
\left[\widehat{L}_k^{\alpha,(q,p)},\widehat{J}_m^{(q,p)}\right]&=k  \widehat{J}_{k+m}^{(q,p)},                             & &k\geq -\alpha,\, m>0,\\
\left[\widehat{L}_k^{\alpha,(q,p)},\widehat{L}_m^{\alpha,(q,p)}\right]&=(k-m)\widehat{L}_{k+m}^{\alpha,(q,p)}, \quad  \quad   & &k,m\geq -\alpha.
\end{aligned}
\end{equation}
\end{theorem}
\begin{proof}
Consider the conjugation of the constraints \eqref{Heis} and \eqref{VirC} by the translation operator 
\be
\exp\left({\hbar^{-1} \sum_{j=1+\alpha} \tilde{v}_{2j+1}^{(\alpha)} \frac{\p}{\p t_{2j+1}}}\right).
\ee
The Heisenberg constraints do not change, and the conjugated Virasoro constraints are given by the operators
\be
\frac{1}{2}\widehat{L}_{2k}-\frac{1}{2\hbar}\sum_{j=2k+1+2\alpha}^\infty[z^j]z^{2k+2}{\mathtt y}_\alpha(h(z))\frac{\p}{\p t_{j}} +\frac{\delta_{k,0}}{16}.
\ee
Then the conjugation by the Virasoro group element $\widehat{V}_{q,p}$ is described by Lemma \ref{lemma32}. Note that the value of $a_2$ in  \eqref{rrr} is provided by \eqref{aa}.
\end{proof}
For the case of linear Hodge integrals ($p=0$, $\alpha=1$) the Heisenberg--Virasoro constraints were derived in \cite{A1}. A more convenient basis of constraints was obtained in \cite{GW}. Constraints in Theorem \ref{propp} provide a generalization of this basis for arbitrary $p$ and $\alpha\in\{0,1\}$.
\begin{remark}
The Heisenberg--Virasoro constraints can also be derived from the Kac--Schwarz description, constructed in \cite{H3_2}.
\end{remark}

The sting equation for $\alpha=1$ has a particularly simple form. We have
\begin{equation}
 \begin{aligned}
 \sigma[-2,m]&=  [z^{m+1}] \frac{1}{f'(z)f(z)}\\
 &= [z^{m+2}](1+\sqrt{p+q}z)(1+q z/\sqrt{p+q})\\
& =\delta_{m,-2}+\frac{2q+p}{\sqrt{p+q}}\delta_{m,-1}+q\delta_{m,0}
\end{aligned}
\end{equation}
and
\be
\chi_1[-2,m]=[z^m]{\mathtt y}_1=[z^m] \frac {\sqrt {p+q}}{p}   \left( \log  \left( 1+\sqrt {p+q}z \right)  -\log  \left(1+\frac{qz}{ \sqrt {p+q}} \right) \right).
\ee
Then
\be
\widehat{L}_{-1}^{1,(q,p)}=\frac{1}{2}\widehat{L}_{-2}+\frac{2q+p}{2\sqrt{p+q}} \widehat{L}_{-1}+\frac{q}{2} \widehat{L}_{0}- \frac{1}{2\hbar}\sum_{m=1}^\infty\chi_1[-2,m] \widehat{J}_{m}
- \frac{1}{48}{\frac {{p}^{2}+pq+{q}^{2}}{p+q}}
\ee
annihilates the tau-function
\be
\widehat{L}_{-1}^{1,(q,p)}  \cdot {\tau}_{q,p}^{(1)}=0.
\ee

From the dimension constraints we have
\begin{equation}
\begin{split}
\left(q\frac{\p}{\p q}+p\frac{\p}{\p p}+\sum_{a=0}^\infty \left(a+\frac{1}{2}\right)T_a\frac{\p}{\p T_a}-\frac{2\alpha+1}{2}\hbar \frac{\p}{\p \hbar}\right)\cdot Z_{q,p}^{(\alpha)}=0.
\end{split}
\end{equation}
Let 
\be\label{Eu3}
E=2 q\frac{\p}{\p q}+2 p\frac{\p}{\p p}+\widehat{L}_0
\ee
be the Euler operator. Then from relations between $Z^{(\alpha)}$'s and $\tau_\alpha$'s we have
\begin{proposition}\label{prop1}
\begin{equation}
\begin{split}
E \cdot \tau_{q,p}^{(\alpha)}= (2\alpha+1)\hbar \frac{\p}{\p \hbar}  \tau_{q,p}^{(\alpha)}.
\end{split}
\end{equation}
\end{proposition}
\begin{remark}
From this proposition it immediately follows that one of the parameters $p$, $q$ and $\hbar$ in the tau-functions $\tau_{q,p}^{(\alpha)}$ is redundant. 
\end{remark}
For the Euler operator \eqref{Eu3} we do not know the relation of the form \eqref{Weq1}.
However, tau-functions ${\tau}_{q,p}^{(\alpha)}$ have a simple cut-and-join description, which leads to the algebraic topological recursion. To derive it we will use \eqref{aneq}, where $\widehat{U}$ is a Virasoro group element described in the next section.

\subsection{Factorizations of the Virasoro group elements }

In this section, we will describe the relation between the tau-functions $\tau_{\alpha}$ and ${\tau}_{q,p}^{(\alpha)}$ given by the Virasoro group element. 
Let us use Lemma \ref{virtos} to substitute the translation operators in \eqref{transn} with the equivalent elements of the Virasoro group,
\be\label{eoe}
\exp\left({\hbar^{-1} \sum_{j=1+\alpha}^\infty \tilde{v}_{2j+1}^{(\alpha)} \frac{\p}{\p t_{2j+1}}}\right)\cdot  \tau_{\alpha}({\bf t})=\widehat{V}^{(\alpha)}_\Delta
\cdot  \tau_{\alpha}({\bf t}).
\ee
Here the Virasoro group operators $\widehat{V}^{(\alpha)}_\Delta$ contain only even generators $\widehat{L}_{2k}$. By Definition \ref{def1} these operators can be described by the formal series 
\be
f_\Delta^{(\alpha)}(z)
=\Xi[\widehat{V}^{(\alpha)}_\Delta ].
\ee
Here from \eqref{vtilde1} we have
\begin{equation}\label{deltf}
\begin{split}
f_\Delta^{(0)}(z)=Y(z), \quad \quad
\frac{f_\Delta^{(1)}(z)^3}{3}= \int_0^z Y(\eta) \eta \hbox{d}\eta,
\end{split}
\end{equation}
where
\be\label{Y}
Y(z)=\frac{1}{2}\left( {\mathtt y}_1(h(z))-{\mathtt y}_1(h(-z))\right)
\ee
is the odd part of ${\mathtt y}_1(h(z))$.
To derive the first equation we use
\be
{\mathtt y}_1(h(z))=\int_0^{h(z)} \frac{f(z) \hbox{d} f(z)}{\eta}=\int_0^z \frac{\eta \hbox{d}\eta}{h(\eta)}.
\ee

Combining \eqref{eoe} with \eqref{le} one gets
\be\label{grel}
\tau_{q,p}^{(\alpha)}({\bf t})=\widehat{V}_{q,p}^{(\alpha)} \cdot  \tau_{\alpha}({\bf t}),
\ee
where
\be
\widehat{V}_{q,p}^{(\alpha)}=\widehat{V}_{q,p} \, \widehat{V}^{(\alpha)}_\Delta \in {\mathcal V}_+
\ee
is an element of the Virasoro group.

Let us find a different, more convenient factorization of this Virasoro group element $\widehat{V}_{q,p}^{(\alpha)} $.
For $p\neq 0$ and $q \neq 0$ we introduce
\begin{equation}\label{hnew}
\tilde{h}(z)=\frac{p}{q\sqrt{p+q}}\frac{\exp\left(2\frac{q}{\sqrt{p+q}}z\right)-1}{1-\exp\left(-2\frac{p}{\sqrt{p+q}}z\right)}-\frac{1}{\sqrt{p+q}}.
\end{equation}
For $p=0$ it degenerates to
\be
\tilde{h}(z)=\frac{\exp\left(2\sqrt{q}z\right)-1}{2qz}-\frac{1}{\sqrt{q}}
\ee
and for $q=0$ it degenerates to
\be
\tilde{h}(z)=\frac{2z}{1-\exp\left(-2\sqrt{p}z\right)}-\frac{1}{\sqrt{p}}.
\ee
Inverse function is given by the series
\begin{equation}
\begin{split}\label{tildef}
\tilde{f}(z)&=z-\frac{1}{3}\,{\frac { p+2\,q}{\sqrt {p+q}}} z^{2}+{
\frac { 2\,{p}^{2}+5\,pq+5\,{q}^{2} }{9(p+q)}} z^{3}
-\frac{2}{135}{\frac {11\,{p}^{3}+39\,{p}^{2}q+51\,p{q}^{2}+34
\,{q}^{3}}{(p+q)^{3/2}}}z^{4}\\
&+{\frac {52\,{p}^{4}+236\,q{p}^{3}+429\,{p}^{2}{q}^{2}+386\,p{q}^{3}+
193\,{q}^{4}}{405\, \left( p+q \right) ^{2}}}z^5+O(z^6).
\end{split}
\end{equation}
It is easy to see that both $\tilde{h}(z)$ and $\tilde{f}(z)$ are formal series in $z$ of the form $z+ \sum_{k=2}^\infty\frac{c_k}{(p+q)^\frac{k-1}{2}} z^k$, where $c_k\in {\mathbb Q}[q,p]$ is a polynomial of degree $k-1$. Set
\be
\widehat{\tilde{V}}=\Xi^{-1}(\tilde{f}(z)).
\ee
Consider an element of the Virasoro group $\widehat{\tilde{V}}^{-1}\widehat{V}_{q,p}$.
Then
\be\label{hc}
\Xi[ \widehat{\tilde{V}}^{-1}\widehat{V}_{q,p}]=
f(\tilde{h}(z))
\ee
is a composition of \eqref{hnew} and \eqref{fpq1}. Let the coefficients ${\bf b}^{(\alpha)}$ be given by
\be\label{tib}
\exp\left({\sum_{k\in{\mathbb Z}_{>0}} {b}_k^{(\alpha)} \widehat{L}_k} \right)=
\widehat{\tilde{V}}^{-1}\,\widehat{V}_{q,p}\,
\widehat{V}^{(\alpha)}_\Delta,
\ee
then the Virasoro group element $\widehat{V}_{q,p}^{(\alpha)}$ can be factorized as
\be\label{gvir1}
\widehat{V}_{q,p}^{(\alpha)}=\widehat{\tilde{V}} \exp\left({\sum_{k\in{\mathbb Z}_{>0}} {b}_k^{(\alpha)} \widehat{L}_k} \right).
\ee
We have
\be
f(z,{\bf {b}^{(\alpha)}})=\Xi(\widehat{\tilde{V}}^{-1}\,\widehat{V}_{q,p}\,
\widehat{V}^{(\alpha)}_\Delta) =f_\Delta^{(\alpha)}(f(\tilde{h}(z))). 
\ee
The reason to introduce the Virasoro group element $\widehat{\tilde{V}}$ is the following 
\begin{lemma}
The function $\Xi[\widehat{\tilde{V}}^{-1}\widehat{V}_{q,p}]$ given by \eqref{hc} is inverse to $Y(z)$, given by \eqref{Y}, or
\be\label{Yinv}
f(\tilde{h}(Y(z)))=z.
\ee
\end{lemma}
\begin{proof}
Assume that $p\neq 0$ and $q\neq 0$. Recall that ${\mathtt x}$ and ${\mathtt y}_1$ satisfy the equation \eqref{spp},
\be
p\exp\left(-q{\mathtt x}(z)\right)=(p+q)\exp\left(\frac{q}{\sqrt{p+q}}{\mathtt y}_1(z)\right)-q\exp\left(\sqrt{p+q}{\mathtt y}_1(z)\right).
\ee
Let us change the variable $z\mapsto h(z)$. We get
\be
p\exp\left(-qz^2/2\right)=(p+q)\exp\left(\frac{q}{\sqrt{p+q}}{\mathtt y}_1(h(z))\right)-q\exp\left(\sqrt{p+q}{\mathtt y}_1(h(z))\right),
\ee
or, if we change the sign of $z$,
\be
p\exp\left(-qz^2/2\right)=(p+q)\exp\left(\frac{q}{\sqrt{p+q}}{\mathtt y}_1(h(-z))\right)-q\exp\left(\sqrt{p+q}{\mathtt y}_1(h(-z))\right).
\ee
Making the inverse transformation $z\mapsto f(z)$, we have
\be
p\exp\left(-q{\mathtt x}(z)\right)=(p+q)\exp\left(\frac{q}{\sqrt{p+q}}\chi\right)-q\exp\left(\sqrt{p+q}\chi\right),
\ee
where we denote 
\be
\chi={\mathtt y}_1(h(-f(z))).
\ee

If we substitute the expression for ${\mathtt x}(z)$ here, we get
\be
p\frac{(1+\sqrt{p+q}z)^\frac{q}{p}}{\left(1+qz/\sqrt{p+q}\right)^\frac{p+q}{p}}=(p+q)\exp\left(\frac{q}{\sqrt{p+q}}\chi\right)-q\exp\left(\sqrt{p+q}\chi\right),
\ee
which, after a straightforward computation leads to the identity
\be
\frac{p}{q\sqrt{p+q}}\frac{\left(\frac{1+\sqrt{p+q}z}{1+qz/\sqrt{p+q}}\right)^\frac{q}{p}\exp\left(-\frac{q}{\sqrt{p+q}}\chi\right)-1}{1-\left(\frac{1+qz/\sqrt{p+q}}{1+\sqrt{p+q}z}\right)
\exp\left(\frac{p}{\sqrt{p+q}}\chi\right)}-\frac{1}{\sqrt{p+q}}=z.
\ee
The left hand side of this equation is equal to $\tilde{h}\left(\frac{1}{2}\left( {\mathtt y}_1(z)-\chi\right)\right)$;
therefore, we have 
\be
\tilde{h}\left(\frac{1}{2}\left( {\mathtt y}_1(z)-{\mathtt y}_1(h(-f(z)))\right)\right)=z.
\ee
If we apply $f(z)$ to the both sides of this equation, we have
\be
f\left(\tilde{h}\left(\frac{1}{2}\left( {\mathtt y}_1(z)-{\mathtt y}_1(h(-f(z)))\right)\right)\right)=f(z).
\ee
After the change of variables $z \mapsto h(z)$ we have
\be
f\left(\tilde{h}\left(\frac{1}{2}\left( {\mathtt y}_1(h(z))-{\mathtt y}_1(h(-z))\right)\right)\right)=z.
\ee
Then the statement of the lemma for $q\neq 0$ and $p \neq 0$ follows from the definition of $Y(z)$, \eqref{Y}.
Special cases with $p=0$ or $q=0$ can be considered similarly as the certain limits.
\end{proof}
From this lemma and \eqref{deltf} it follows that the Virasoro group elements \eqref{tib} are described by the functions
\begin{equation}\label{fbtild}
\begin{aligned}
f(z;{\bf {b}}^{(0)})&=Y(f(\tilde{h}(z)))=z,\\
\frac{f(z;{\bf {b}}^{(1)})^3}{3}&=\int_0^z \eta\, \hbox{d} {\mathtt x}(\tilde{h}(\eta)).
\end{aligned}
\end{equation}
In particular, ${\bf {b}}^{(0)}={\bf 0}$ and $\widehat{\tilde{V}}^{-1}\,\widehat{V}_{q,p}\,
\widehat{V}^{(0)}_\Delta=1$, therefore, $\widehat{V}_{q,p}^{(0)}=\widehat{\tilde{V}}$.

The last integral in \eqref{fbtild} can be computed explicitly as a series,
\be\label{fbb}
\frac{f(z;{\bf {b}}^{(1)})^3}{3}=\sum_{k=1}^{\infty}\frac{4^k B_{2k}}{(2k+1)!}\frac{(p+q)^{2k+1}-q^{2k+1}-p^{2k+1}}{pq(p+q)^k}z^{2k+1}.
\ee
Here $B_{2k}$ are the Bernoulli numbers defined by
\be
\frac{ze^z}{e^z-1}=1+\frac{z}{2}+\sum_{k=1}^\infty \frac{B_{2k}z^{2k}}{(2k)!}.
\ee

\begin{remark}
For $p=0$ we have
\be
\left.\frac{f(z;{\bf {b}}^{(1)})^3}{3}\right|_{p=0}=\frac{z^2}{\sqrt{q}}\coth(\sqrt{q}z)-\frac{z}{q^{\frac{5}{2}}}
\ee
which, up to inversion and rescaling of the variable, coincides with the function describing one of the components of the Virasoro group element in \cite[Section 2.2]{A1}. This observation indicates that the factorization \eqref{gvir1} is a generalization for $p\neq 0$ of the factorization of the group element considered in \cite{A1}.
\end{remark}

\begin{proposition}\label{cor1}
The Virasoro group element $\widehat{V}_{q,p}^{(\alpha)}$ by Definition \ref{def1} corresponds to the series 
\be
f_\alpha(z)=\Xi(\widehat{V}_{q,p}^{(\alpha)})=f(\tilde{f}(z);{\bf {b}^{(\alpha)}}),
\ee
where $f(z,{\bf {b}}^{(\alpha)})$ and  $\tilde{f}$ are given by  \eqref{fbtild} and the function inverse to \eqref{hnew}. We have
\be
f_0(z)=\tilde{f}(z),\quad \quad \frac{f_1(z)^3}{3}=\int_{0}^z \tilde{f}(\eta)  {\hbox{d}}  {\mathtt x}(\eta).
\ee
\end{proposition}

\subsection{Cut-and-join description of the triple Hodge integrals}

The cut-and-join description of the tau-function for the triple Hodge integrals satisfying the Calabi--Yau condition follows from the cut-and-join description for the KW and BGW tau-functions, see Section \ref{SHV}.
Let
\be
\widehat{{W}}_{\alpha}^{q,p}=\frac{1}{2\alpha+1} \sum_{k=0}^\infty  \widehat{J}_{-2k-1}^{\alpha}\left(\widehat{L}_{2k-2\alpha}^{\alpha}+\frac{\delta_{k,\alpha}}{8}\right).
\ee
Here
\begin{equation}
 \begin{aligned}
 \widehat{J}_{k}^{\alpha}&=\widehat{V}_{q,p}^{(\alpha)} \widehat{J}_{k} \left(\widehat{V}_{q,p}^{(\alpha)}\right)^{-1},\\
    \widehat{L}_{k}^{\alpha}&=\widehat{V}_{q,p}^{(\alpha)} \widehat{L}_{k} \left(\widehat{V}_{q,p}^{(\alpha)}\right)^{-1},
\end{aligned}
\end{equation}
are the conjugated operators of the Heisenberg--Virasoro algebra. By Lemma \ref{lemma32} they are given by the linear combinations
\begin{equation}
 \begin{aligned}
 \widehat{J}_{k}^{\alpha}& = \sum_{m=k}^\infty \rho_{\alpha}[k,m] \widehat{J}_{m},                                                     & &k\in {\mathbb Z},\\
    \widehat{L}_{k}^{\alpha}&= \sum_{m=k}^\infty \sigma_{\alpha}[k,m] \widehat{L}_{m}+c_\alpha\delta_{k, -2},   \quad \quad & &k\geq -2,
\end{aligned}
\end{equation}
where
\be
\rho_\alpha[k,m]=  [z^m] f_\alpha(z)^k,\quad \quad \sigma_\alpha[k,m]=  [z^{m+1}]\frac{f_\alpha(z)^{k+1}}{f'_\alpha(z)}
\ee
with $f_\alpha$ given by Proposition \ref{cor1} and
\be
c_\alpha=\left(\frac{\delta_{\alpha,1}}{90}-\frac{1}{18}\right)\frac{p^2+pq+q^2}{p+q}.
\ee

Because of a certain arbitrariness of the cut-and-join description of the tau-functions $\tau_\alpha$, see Remark \ref{remp}, we can also construct different versions of the cut-and-join operators for the tau-functions $\tau_{q,p}^{(\alpha)}$. For example,  consider the operators $\widehat{W}_\alpha^{*q,p}$ given by \eqref{Wpri} with $\widehat{J}_k^a$ substituted by $ \widehat{J}_{k}^{\alpha}$.

\begin{theorem}\label{T1}
\begin{equation}
\begin{split}\label{CAJqp}
 {\tau}_{q,p}^{(\alpha)}&=\exp\left(\hbar\, \widehat{W}_\alpha^{q,p}\right)\cdot 1\\
 &=\exp\left(\hbar\, {\widehat{W}_\alpha^{*q,p}}\right)\cdot 1.
\end{split}
\end{equation}
\end{theorem}
\begin{proof}
The statement immediately follows from Lemma \ref{lemma3.3}, Lemma  \ref{l35}, and Proposition \ref{cor1}.
\end{proof}
Operators $ \widehat{W}_\alpha^{q,p}$ and $ {\widehat{W}_\alpha^{*q,p}}$ do not belong to the $\gl(\infty)$ algebra of KP hierarchy symmetries.

Expressions for the coefficients $\rho_\alpha$ and $\sigma_\alpha$ can be simplified. Namely, for $\alpha=0$, since  $f_0(z)=\tilde{f}(z)$, we have
\begin{equation}
\begin{split}\label{BGWdef}
\rho_0[k,m]&= [z^m]\tilde{f}(z)^k=\frac{1}{2 \pi i} \oint 
 \frac{\hbox{d} \tilde{h}(z)}{\tilde{h}(z)^{m+1}} z^k,\\
 \sigma_0[k,m]&= [z^{m+1}] \frac{\tilde{f}(z)^{k+1}}{\tilde{f}'(z)}= \frac{1}{2 \pi i} \oint \frac{\tilde{h}'(z)\hbox{d}\tilde{h}(z)}{\tilde{h}(z)^{m+2}} z^{k+1}.
\end{split}
\end{equation}
Both $(p+q)^{(m-k)/2}\rho_0[k,m]\in {\mathbb Q}[q,p]$ and $(p+q)^{(m-k)/2}\sigma_0[k,m]\in {\mathbb Q}[q,p]$ are polynomials of the total degree $m-k$.

Let us expand the cut-and-join operator
\be
\widehat{{W}}_{0}^{q,p}=\sum_{k=0}^\infty \widehat{{W}}_{0,1-k}^{q,p},
\ee
where $\deg \widehat{{W}}_{0,k}^{q,p}=k$. The leading term coincides with the cut-and-join operator for the BGW model (\ref{leadcaj}), $\widehat{{W}}_{0,1}^{q,p}=\widehat{W}_0$.
For the sub-leading term
\be
\widehat{{W}}_{0,0}^{q,p}=\sum_{k=0}^\infty\left(\rho_0[-2k-1,-2k]\widehat{J}_{-2k}\widehat{L}_{2k}+\sigma_0[2k,2k+1]\widehat{J}_{-2k-1}\widehat{L}_{2k+1}\right)
\ee
it is possible to find corresponding coefficients $\rho_0$ and $\sigma_0$ explicitly. From \eqref{BGWdef} and \eqref{hnew} for the non-negative $k$ we immediately get
\be
\rho_0[-2k-1,-2k]=\frac{2k+1}{3}\frac{p+2q}{\sqrt{p+q}},\quad\quad\sigma_0[2k,2k+1]=-\frac{2k-1}{3}\frac{p+2q}{\sqrt{p+q}}.
\ee
Therefore, we have
\begin{equation}
\begin{split}
\widehat{{W}}_{0,0}^{q,p}&=\frac{1}{3}\frac{p+2q}{\sqrt{p+q}}\sum_{k=0}^\infty\left((2k+1)\widehat{J}_{-2k}\widehat{L}_{2k}-(2k-1)\widehat{J}_{-2k-1}\widehat{L}_{2k+1}\right).
\end{split}
\end{equation}

In the same way for the next term we have
\begin{equation}
\begin{split}
\widehat{{W}}_{0,-1}^{q,p}&=\sum_{k=0}^\infty\left(\rho_0[-2k-1,-2k+1]\widehat{J}_{-2k+1}\left(\widehat{L}_{2k}+\frac{\delta_{k,0}}{8}\right)\right.\\
&\left.+\rho_0[-2k-1,-2k]\sigma_0[2k,2k+1]\widehat{J}_{-2k}\widehat{L}_{2k+1}+\sigma_0[2k,2k+2]\widehat{J}_{-2k-1}\widehat{L}_{2k+2}\right),
\end{split}
\end{equation}
where
\begin{equation}
\begin{split}
\rho_0[-2k-1,-2k+1]&=-\frac{2k+1}{9}\left((k-1)\frac{p^2}{p+q}+(4k-1)q\right),\\
 \sigma_0[2k,2k+2]&=\frac{1}{9}\left( \left( 2\,{k}^{2}+k-2 \right){\frac { {p}^{2}}{p+q}}+ \left( 8{k}^{2}-2k-2 \right) q\right).
\end{split}
\end{equation}

For $\alpha=1$ the coefficients $\rho_1$ and $\sigma_1$ can be expressed in terms of the residues of the combinations of functions $\tilde{h}(z)$ and $f(z;{\bf {b}}^{(1)})$, given by \eqref{hnew} and \eqref{fbb}, namely
\begin{equation}
\begin{split}
\rho_1[k,m]&= [z^m] f_1(z)^k=\frac{1}{2 \pi i} \oint 
 \frac{\hbox{d} \tilde{h}(z)}{\tilde{h}(z)^{m+1}}  f(z;{\bf {b}}^{(1)})^k,\\
 \sigma_1[k,m]&=   [z^{m+1}]\frac{f_1(z)^{k+1}}{f'_1(z)}= \frac{1}{2 \pi i} \oint \frac{\tilde{h}'(z)\hbox{d}\tilde{h}(z)}{\tilde{h}(z)^{m+2}}  \frac{f(z;{\bf {b}}^{(1)})^{k+1}}{f'(z;{\bf {b}}^{(1)})}.
\end{split}
\end{equation}

By construction, the cut-and-join operators $\widehat{W}_\alpha^{q,p}$ and $\widehat{W}_\alpha^{*q,p}$ are cubic in $\widehat{J}_k$.  Moreover, they do not depend on $\hbar$. Hence, the cut-and-join formulas (\ref{CAJqp}) describe the algebraic topological recursion. Namely, if we consider the topological expansion of the tau-functions
\be
 {\tau}_{q,p}^{(\alpha)}=\sum_{k=0}^\infty \hbar^{k} {\tau}_{q,p}^{(\alpha,k)},
\ee
then the coefficients of this expansion satisfy
\begin{equation}
\begin{split}
{\tau}_{q,p}^{(\alpha,k)}&=\frac{1}{k} \widehat{W}_\alpha^{q,p} \cdot {\tau}_{q,p}^{(\alpha,k-1)}\\
&=\frac{1}{k} \widehat{W}_\alpha^{*q,p} \cdot {\tau}_{q,p}^{(\alpha,k-1)}
\end{split}
\end{equation}
with the initial condition ${\tau}_{q,p}^{(\alpha,0)}=1$.
The operators $\widehat{W}_\alpha^{q,p} $ and $\widehat{W}_\alpha^{*q,p} $ contain infinitely many terms, however, only a finite number of them contribute to each step of recursion. Moreover, for a contributions with the fixed numbers of marked points we have a recursion similar to \eqref{fixn}.
In Appendix \ref{AA} we provide the first few terms $F_k^\alpha({\bf t})$ of the expansion of $\log{\tau}_{q,p}^{(\alpha)}$,
\be
{\tau}_{q,p}^{(\alpha)}({\bf t})=\exp\left(\sum_{k=1}^\infty \hbar^k F_k^\alpha({\bf t})\right)
\ee
obtained by this recursion with the cut-and-join operators $\widehat{W}_\alpha^{q,p}$ using Maple.

For $q=0$ or $p=0$ we get the cut-and-join description for the tau-function of the linear Hodge integrals.

\begin{remark}
Using this approach one can also restore other types of cut-and-join-like representations of the generating functions of the triple Hodge integrals. For example, it is possible to find a description similar to the one constructed in \cite{Lapl} for the KW tau-function.
\end{remark}

\begin{remark}
We believe that there is a universal distinguished choice of the cut-and-join operators associated to some natural representation theory interpretation. It may be related to the quantum Airy structures \cite{Boro1,Boro2}. 
\end{remark}

\section{KdV reduction}\label{S5}

For $p=-2q$ the tau-functions $\tau^{(\alpha)}_{q,p}$ do not depend on even times, hence, they are solutions of the KdV hierarchy, a 2-reduction of the KP hierarchy \cite{H3_1}. These tau-functions have particularly nice properties, and can be identified with the KW and BGW tau-functions with shifted arguments.
We will investigate them in this and the next sections. 
\begin{remark}
In \cite{DLYZ} it was shown that in this case the tau-function $\tau^{(1)}_{q,-2q}$ is related to the discrete KdV hierarchy (aka the
Volterra lattice hierarchy). We expect that $\tau^{(0)}_{q,-2q}$ is also related to the Volterra hierarchy in the similar way.
\end{remark}

\subsection{Translations and \texorpdfstring{$\kappa$}--classes}

With the forgetful map $\pi: \overline{\mathcal M}_{g,n+1} \rightarrow \overline{\mathcal M}_{g,n} $ we define the {\em Miller--Morita--Mumford tautological classes} \cite{Mumford}, $\kappa_k:= \pi_* \psi_{n+1}^{k+1} \in H^{2k}(\overline{\mathcal{M}}_{g,n},\mathbb{Q})$. According to Manin and Zograf \cite{MZ}, insertion of these classes into the intersection numbers \eqref{eq:products} can be described by the translation of the variables $t_k$ responsible for  the insertion of the $\psi$-classes. Let us consider the generating functions of the intersection numbers of $\psi$ and $\kappa$ classes with (for $\alpha=0$) and without (for $\alpha=1$) Norbury's $\Theta$-classes
\be
F_{g,n}^\alpha({\bf T}, {\bf s})=\sum_{a_1,\ldots,a_n\ge 0}\int_{\overline{\mathcal{M}}_{g,n}} \Theta_{g,n}^{1-\alpha} e^{\sum_{j=1}^\infty s_j \kappa_j}\psi_1^{a_1} \psi_2^{a_2} \cdots \psi_n^{a_n} \frac{\prod T_{a_i}}{n!}
\ee
and
\be
\tau_\alpha({\bf t}, {\bf s})=\left.\exp\left(\sum_{g=0}^\infty \sum_{n=0}^\infty \hbar^{2g-2+n}F_{g,n}^\alpha({\bf T}, {\bf s})\right)\right|_{T_k=(2k+1)!! t_{2k+1}}.
\ee

We introduce the polynomials $q_j({\bf s})$, defined by the generating function
\be\label{genss}
1-\exp\left(-\sum_{j=1}^\infty s_j z^j\right)=\sum_{j=1}^\infty q_j({\bf s})z^j,
\ee
and we put $q_k({\bf s})=0$ for $k<1$.
For $k \geq 1$ they are nothing but the negative of the {\em elementary Schur functions} 
\be
q_j({\bf s})=-p_j(-{\bf s}),
\ee
where
\be
\exp\left(\sum_{j=1}^\infty s_j z^j\right)=\sum_{j=0}^\infty p_j({\bf s})z^j.
\ee
Then the generating functions $\tau_\alpha({\bf t}, {\bf s})$ can be obtained from the generating functions $\tau_\alpha({\bf t})$ without $\kappa$-classes, see Section \ref{S31}, by translation
\begin{theorem}[\cite{MZ,NorbS}]\label{TMZN}
For $\alpha\in\{0,1\}$
\be
\tau_\alpha({\bf t}, {\bf s})=\tau_\alpha\left(\left\{ t_{2k+1}+\frac{1}{\hbar}\frac{q_{k-\alpha}({\bf s})}{(2k+1)!!}\right\}\right). 
\ee
\end{theorem}

\subsection{KdV reduction for triple Hodge integrals}

From the change of the variables \eqref{changeof} it is clear that for $p=-2q$ the tau-functions $\tau_{q,p}^{(\alpha)}({\bf t})$ do not depend on even variables $t_{2k}$, therefore they are tau-function of the KdV hierarchy. Let us put $q=-u^2$ with new parameter $u$, so that $p=2u^2$. This parametrization corresponds to the insertion of the triple Hodge class of the form $\Lambda(-2u^2)\Lambda(-2u^2)\Lambda(u^2)$.
\begin{remark} From Proposition \ref{prop1} it follows that the parameter $u$ is redundant. However, we will leave it as a free parameter to make the limit $u=0$ more transparent. 
\end{remark}
In this case the series $\tilde{f}(z)$, see \eqref{tildef}, have a simple explicit form
\begin{equation}
\begin{split}
\tilde{f}(z)&=\frac{1}{2u}\log\left(\frac{1+uz}{1-uz}\right)\\
&=\frac{1}{u}\arctanh(uz)\\
&=\sum_{k=0}^\infty\frac{u^{2k} z^{2k+1}}{2k+1}.
\end{split}
\end{equation}
It is an odd function. The Virasoro group operators $\widehat{V}^{(\alpha)}_{q,p}$ in \eqref{grel}, relating $\tau_{\alpha}$ to the tau-functions for the triple Hodge integrals, contain only even positive components of the Virasoro operators. Therefore, using Lemma \ref{virtos} one equivalently describes them by translation operators
\be
\tau_{-u^2,2u^2}^{(\alpha)}({\bf t})= \exp\left({\hbar^{-1} \sum_{j=1+\alpha} v_{2j+1}^{(\alpha)}(u) \frac{\p}{\p t_{2j+1}}}\right) \cdot  \tau_{\alpha}({\bf t}).
\ee
Let us find the coefficients $v_{2j+1}^{(\alpha)}(u)$ using Lemma \ref{virtos} and Proposition \ref{cor1}.

For $\alpha=0$ these coefficients are given by the generating function
\be
\sum_{k=1}^\infty v_{2k+1}^{(0)}(u)z^{2k+1}=z-\tilde{f}(z),
\ee
therefore,
\be\label{shift0}
v_{2k+1}^{(0)}(u)=-\frac{u^{2k}}{2k+1}.
\ee

By Theorem \ref{TMZN} this translation of the variables corresponds to the insertion of the $\kappa$-classes with parameters $s^0_k$ given by $q_k({\bf s}^0)=-(2k-1)!! u^{2k}$. From \eqref{genss} we have
\be\label{ss0}
\sum_{k=1}^\infty s_k^0 z^k=-\log\left(1+\sum_{k=1}^\infty(2k-1)!! \left(zu^2\right)^k\right).
\ee
For instance, $s_1^0=-u^2$, $s_2^0=-\frac{5}{2}u^4$, $s_3^0=-\frac{37}{3}u^6$.
In general, $s_k^0=-\frac{a_k}{k}u^{2k}$, where the sequence $a_k$ is described in \cite[A004208]{oeis}.

To find the coefficients for $\alpha=1$ we need the specialization of \eqref{xdd},
\be
 \hbox{d} {\mathtt x}= \frac{z dz}{1- u^2 z^2}.
\ee
From Proposition \ref{cor1} we have
\begin{equation}
\begin{split}
\frac{1}{3}f_1(z)^3&=\int_0^z \tilde{f}(\eta) \hbox{d} {\mathtt x}(\eta)\\
&=\frac{1}{u} \int_0^z\arctanh(u\eta) {\frac {\eta  \hbox{d} \eta}{1-u^2{\eta}^{2}}}\\
&=\frac{1}{2u^2}\int_0^z \eta \left( \frac{\p}{\p \eta}\arctanh(u\eta)^2\right)  \hbox{d} \eta.
\end{split}
\end{equation}
From Lemma \ref{virtos} we know that the coefficients satisfy
\be
\sum_{k=1}^\infty v_{2k+1}^{(1)}(u) z^{2k+1}=\frac{1}{3}\left(z^3-f_1(z)^3\right),
\ee
therefore,
\be\label{shift1}
v_{2k+3}^{(1)}(u) =-\frac{u^{2k}}{2k+3}\sum_{j=0}^k\frac{1}{2j+1}.
\ee
Here we use the series expansion of $\arctanh(z)^2$, see e.g. \cite[A004041]{oeis}.

Using Theorem \ref{TMZN} one can find the values of the parameters $s_k$, associated to these translations,
\begin{equation}\label{ss1}
\begin{split}
\exp\left(-\sum_{j=1}^\infty s_j^1 z^j\right)&=1-\sum_{k=1}^\infty (2k+3)!! v_{2j+3}^{(1)}(u) z^k\\
&=\sum_{k=0}^\infty (2k+1)!! \left(\sum_{j=0}^k\frac{1}{2j+1}\right)\left(zu^2\right)^k.
\end{split}
\end{equation}
In particular, $s_1^1=-4u^2$, $s_2^2=-15u^4$, $s_3^1=-\frac{316}{3}u^6$.

Therefore, for ${\bf s}^{\alpha}$ given by \eqref{ss0} and \eqref{ss1} we have the following relation between the generating functions of different types of intersection numbers:
\begin{theorem}\label{taueq}
The generating functions of the triple Hodge integrals satisfying Calabi--Yau condition for $p=-2q=2u^2$ coincide with the generation functions of the intersection numbers containing $\kappa$ classes, 
\be\label{kdvid}
\tau^{(\alpha)}_{-u^2,2u^2}({\bf t})=\tau_\alpha({\bf t},{\bf s^\alpha}).
\ee
\end{theorem}
Here the tau-functions $\tau^{(\alpha)}$ are the generating functions of triple Hodge integrals satisfying Calabi--Yau condition \eqref{tau^alpha} in the variables \eqref{changeof}.

\begin{remark}
Geometric interpretation of this identity is not clear yet. It would be interesting to find if it follows from the relations between the cohomology classes on ${\overline{\mathcal{M}}_{g,n}}$.
\end{remark}

\begin{remark}
After this paper was uploaded to arXiv, Theorem \ref{taueq} for $\alpha=1$ was proven using a different approach in \cite{YZ}.
\end{remark}

This theorem implies that for investigation of the tau-functions $\tau^{(\alpha)}_{-u^2,2u^2}({\bf t})$ we can use known properties of the KW and BGW tau-functions. In particular, the tau-functions $\tau^{(\alpha)}_{-u^2,2u^2}({\bf t})$ obey the Virasoro constraints, obtained from the constraints \eqref{VirC} by a translation of variables 
$t_k\mapsto t_k +\hbar^{-1}v_{2j+1}^{(\alpha)}(u)$. This basis in the space of constraints can be more convenient when one, constructed in Section \ref{S42}.

Let us consider the identification of the intersection numbers for the given genus $g$ and number of the marked points $n$. Using the change of variables \eqref{changeof} we can reformulate Theorem \ref{taueq} in terms of these intersection numbers. Namely, for $p=-2q=2u^2$ the change of the variables \eqref{changeof} is given by
\be
T^{-u^2,2u^2}_k({\bf t})=\sum_{m=1}^\infty m t_m\left(\frac{\p}{\p t_{m-2}}-u^2\frac{\p}{\p t_m}\right)T^{-u^2,2u^2}_{k-1}({\bf t})
\ee
with $T^{-u^2,2u^2}_0=t_1$. 
It is generated by 
\be
\frac{\p}{\p {\mathtt x}}=\frac{1-u^2z^2}{z} \frac{\p}{\p z}.
\ee
Namely, the functions \eqref{phik},
\be\label{Phi}
\Phi_k(z):=\left(-\frac{\p}{\p {\mathtt x}}\right)^k\cdot \frac{1}{z},
\ee
define a change of variables, if we associate $T_k=(2k+1)!! t_{2k+1}$ with $(2k-1)!! z^{-2k-1}$ and $ T^{-u^2,2u^2}_k$ with $\Phi_k(z)$. 

For $g,n\geq 0$ let us introduce the {\em correlation functions}
\be\label{corf}
W_{g,n}^\alpha(z_1,\dots,z_n)=\sum_{a_1,\dots,a_n\ge 0} \int_{\overline{\mathcal{M}}_{g,n}}
\Theta_{g,n}^{1-\alpha} \Lambda(-2u^2)\Lambda(-2u^2)\Lambda(u^2)\prod_{j=1}^n \psi_j^{a_j}
\Phi_{a_j}(z_j)
\ee
and
\be\label{corsh}
\tilde{W}_{g,n}^\alpha(z_1,\dots,z_n)=\sum_{a_1,\dots,a_n\ge 0} \int_{\overline{\mathcal{M}}_{g,n}}
\Theta_{g,n}^{1-\alpha} e^{\sum_{j=1}^\infty s_j^\alpha \kappa_j} \prod_{j=1}^n \psi_j^{a_j}
\left(-\frac{1}{z_j}\frac{\p}{\p z_j}\right)^{a_j} z_j^{-1}.
\ee
Then the following statement is equivalent to Theorem \ref{taueq}:
\begin{corollary}\label{cor}
For $g,n\geq 0$
\be\label{Wid}
W_{g,n}^\alpha(z_1,\dots,z_n)=\tilde{W}_{g,n}^\alpha(z_1,\dots,z_n).
\ee
\end{corollary}
For $u=0$ this relation is trivial. Note that for the non-stable cases with $2g-2+n\leq 0$ both ${W}_{g,n}^\alpha$ and $\tilde{W}_{g,n}^\alpha$ vanish identically.
For $n=0$ we have
\be\label{eqq}
\int_{\overline{\mathcal{M}}_{g}} \Theta_{g,0}^{1-\alpha} \Lambda_g (-2) \Lambda_g (-2) \Lambda_g (1)=\int_{\overline{\mathcal{M}}_{g}} \Theta_{g,0}^{1-\alpha}  e^{\sum_{j=1}^\infty s_j^\alpha \kappa_j}.
\ee
For $\alpha=1$ according to Faber and Pandharipande \cite{FP} one has
\be
\int_{\overline{\mathcal M}_{g}}  \Lambda_g (-q) \Lambda_g (-p) \Lambda_g (\frac{pq}{p+q})=\frac{1}{2}\left(\frac{q^2p^2}{q+p}\right)^{g-1}
\frac{B_{2g}}{2g}\frac{B_{2g-2}}{2g-2}\frac{1}{(2g-2)!}
\ee
for $g\geq 2$. 
For $p=-2q=2u^2$ it reduces to
\be
\int_{\overline{\mathcal M}_{g}}  \Lambda_g (-2u^2) \Lambda_g (-2u^2) \Lambda_g (u^2)=2^{2g-3}u^{6g-6}
\frac{B_{2g}}{2g}\frac{B_{2g-2}}{2g-2}\frac{1}{(2g-2)!}.
\ee
It would be interesting to derive this expression using the right hand side of \eqref{eqq}. 

\begin{remark}
Theorem \ref{taueq} and Corollary \ref{cor} have direct generalizations for the case where the triple Hodge class in \eqref{FF} is deformed by any combination of the Miller--Morita--Mumford tautological classes $e^{\sum_{j=1}^\infty s_j \kappa_j}$ with arbitrary $s_j$. Indeed, such a deformation corresponds to the translation of the time variables, it does not break the KdV integrability, and can be described by the corresponding (in general, different)
deformation by the Miller--Morita--Mumford tautological classes of the right hand sides in \eqref{kdvid} and \eqref{Wid}.
\end{remark}

\section{Tau-functions identification and topological recursion}\label{S6}
In this section, we will describe the Chekhov--Eynard--Orantin topological recursion for the triple Hodge integrals satisfying the Calabi--Yau condition and consider Theorem \ref{taueq} from the point of view of this recursion. 

\subsection{Chekhov--Eynard--Orantin topological recursion}

The Chekhov--Eynard--Orantin topological recursion is a universal construction, which allows to recover an infinite tower of correlation function starting from a (non necessarily globally defined) {\em spectral curve} $\Sigma$, two meromorphic differentials on it, $\hbox{d} x$ and $\hbox{d} y$,  and a canonical bilinear differential $B(z_1,z_2)$ for $z_1,z_2 \in \Sigma$. We call the set
\be
S=\left(\Sigma, B, x, y\right)
\ee
the {\em topological recursion data}.  Let us briefly review the formalism of topological recursion, see \cite{EO,EO1,CN} for more details.

In this paper we restrict ourselves to the spectral curves of genus zero, $\Sigma=\mathbb{C}\mathrm{P}^1$. 
Consider $\mathbb{C}\mathrm{P}^1$ with a fixed global coordinate $z$, then we can specify the spectral curve by the functions $x(z)$ and $y(z)$. Consider  the differential $\hbox{d} x $  with an assumption that $\hbox{d} x$ is a rational differential with simple critical points $p_1,\dots,p_N$. On $\mathbb{C}\mathrm{P}^1$ the canonical bilinear differential is the Cauchy differentiation kernel
\be\label{Cauch}
B_C(z_1,z_2)=\frac{\hbox{d} z_1 \hbox{d} z_2}{(z_1-z_2)^2}.
\ee
We assume that the function $y$ is meromorphic in the neighbourhoods of the zeros of $\hbox{d} x $. We allow the spectral curve to be {\em irregular} \cite{CN}, that is, the poles of $y$ may coincide with the zeroes of  $\hbox{d} x$.

With this input, on the Cartesian product $\Sigma^{n}$ we construct a system of symmetric differentials (or {\em correlators}) $\omega_{g,n}$, $g\geq 0$, $n\geq 1$,  given by 
\begin{align}
	& \omega_{0,1}(z_1) =  y(z_1)  \hbox{d} x (z_1) ; 
	\\ \notag
	& \omega_{0,2}(z_1,z_2) = B_C(z_1,z_2);
\end{align}
and for $2g-2+n>0$ we use the recursion
\begin{align}\label{TR}
	\omega_{g,n} (z_1,\dots,z_n) \coloneqq \ & \frac 12 \sum_{i=1}^N \res_{z\to p_i} 
	\frac{\int^z_{\sigma_i(z)} B_C(z_1,\cdot)} {  \omega_{0,1}(z_1) -\omega_{0,1}(\sigma_i(z_1))  }\Bigg(
	\omega_{g-1,n+1}(z,\sigma_i(z),z_{\llbracket n \rrbracket \setminus \{1\}})
	\\
	\notag
	& + \sum_{\substack{g_1+g_2 = g, I_1\sqcup I_2 = {\llbracket n \rrbracket \setminus \{1\}} \\
			(g_1,|I_1|),(g_2,|I_2|) \not= (0,0) }} \omega_{g_1,1+|I_1|}(z,z_{I_1})\omega_{g_2,1+|I_2|}(\sigma_i(z), z_{I_2})\Bigg),
\end{align}
where  $\sigma_i$ is the deck transformation of $x$ near $p_i$, $i=1,\dots, N$. For the stable cases $2g-2+n>0$ these differentials are meromorphic. We wouldn't go into the discussion of this version of this topological recursion, for more detail e.g. \cite{EO}. As the spectral curve $\Sigma=\mathbb{C}\mathrm{P}^1$ and the canonical bilinear differential $B=B_C$ are the same for all considered cases, we will denote the topological recursion data by $S=\left( x(z), y(z)\right)$, and also call it the spectral curve. 

The $g$th {\em symplectic invariant} for $g\geq 2$ associated to the spectral curve $S=\left( x, y\right)$ is defined by
\be
F^{(g)}=\frac{1}{2-2g} \sum_{i=1}^N \res_{z\to p_i}  \Phi(z) \omega_{g,1}(z),
\ee
where $ \Phi(z)=\int^z  y  \hbox{d} x  $ is a primitive of $y(z)  \hbox{d} x(z)$. Definitions of $F^{(0)}$ and $F^{(1)}$ can be found in \cite{EO}. We will identify $\omega_{g,0}=F^{(g)}$. All correlators $\omega_{g,n}$
 with $ 2- 2g- n < 0$ are called {\em stable}, and the others are called unstable.

The symplectic invariants are expected to be invariant under the symplectic transformations which preserve the symplectic form $ \hbox{d}x\wedge  \hbox{d}y$. However, the invariance under the exchange transformation $x \leftrightarrow y$ is still challengeable. It is known that in general the correlators $\omega_{g,n}$ are not invariant under the  symplectic transformations. Nevertheless, there are some symplectic transformations which leave them invariant. Let us consider an interesting class of such transformations. 

Namely, consider the transformation \cite{EO1} of the form 
\be\label{symptr}
\tilde{x}=g(x), \quad \quad \quad \tilde{y}=\frac{1}{g'(x)}y,
\ee
where $g(x)$ is an analytic function near the zeros of $\hbox{d} x$. This transformation preserves the symplectic form $ \hbox{d} x \wedge  \hbox{d} y= \hbox{d} \tilde{x} \wedge  \hbox{d} \tilde{y}$, therefore, it is symplectic. We assume that both spectral curves $S=(x,y)$ and $\tilde{S}=(\tilde{x},{\tilde{y}})$ are of genus zero. 
\begin{proposition}\label{Props}
If the zeroes of $ \hbox{d} \tilde{x}$ are simple and the set of these zeroes coincides with the set of the zeroes of $ \hbox{d}  x$, then for the two spectral curves $S=(x,y)$ and $\tilde{S}=(\tilde{x}, \tilde {y})$ we have
\be
\omega_{g,n} (z_1,\dots,z_n)=\tilde{\omega}_{g,n} (z_1,\dots,z_n)
\ee
for all stable cases.
\end{proposition}
\begin{proof} For $n=0$  the statement was established in \cite{EO1}. The canonical bilinear differentials for two spectral curves are the same and coincide with the Cauchy differentiation kernel \eqref{Cauch}. Since $x(\sigma_i(z))=x(z)$ in the neighbourhood of $p_i$, we have $\tilde{x}(\sigma_i(z))=g({x}(\sigma_i(z)))=\tilde{x}(z)$. Therefore, $\tilde{\sigma}_i(z)=\sigma_i(z)$ and
\be
 \omega_{0,1}(z_1) -\omega_{0,1}(\sigma_i(z_1)) =\tilde{ \omega}_{0,1}(z_1) -\tilde{\omega}_{0,1}(\tilde{\sigma}_i(z_1)). 
\ee
We see that the initial conditions and the kernel of the topological recursion for $S$ and $\tilde{S}$ are the same, which completes the proof.
\end{proof}

With a topological recursion data we associate a generating function made of only stable correlators with neglected unstable terms,
 \be
 \exp\left(\sum_{2g-2+n>0}\frac{\hbar^{2g-2+n}}{n!}\left.\omega_{g,n}(z_1,\dots,z_n)\right|_{ \Psi_k(z_i)=T_k}\right),
 \ee
 where $\Psi_k(z)= \hbox{d}\left(-\frac{\p}{\p x}\right)^kz^{-1}$.
 
From now on we consider only the case $N=1$, that is, the differential $\hbox{d} x$ has a unique zero.
The KW tau-function $\tau_1({\bf t})$ is described by the topological recursion on the Airy curve $x=\frac{1}{2}y^2$ \cite{EO1}
\be
S=\left(\frac{1}{2}z^2, z\right).
\ee
It was conjectured in \cite{KS2} and proven in \cite{DN} that the BGW tau-function $\tau_0({\bf t})$ is described by the topological recursion on the Bessel curve $xy^2=\frac{1}{2}$ with the topological recursion data
\be
S=\left( \frac{1}{2}z^2, \frac{1}{z}\right).
\ee
For these cases the stable correlators of topological recursion are related to the intersection numbers 
\be
\omega_{g,n}^{\alpha,0}(z_1,\dots,z_n)=\hbox{d}_1\dots \hbox{d}_n \sum_{a_1,\dots,a_n\ge 0} \int_{\overline{\mathcal{M}}_{g,n}}
\Theta_{g,n}^{1-\alpha}\prod_{j=1}^n \psi_j^{a_j}
\left(-\frac{1}{z_j}\frac{\p}{\p z_j}\right)^{a_j} z_j^{-1}.
\ee
for $\alpha=1$ and $\alpha=0$ respectively.

For these two tau-functions translation of variables ${\bf t}$ leads to the change of $y$ described by Proposition 4.8 of Chekhov--Norbury:
\begin{proposition}[\cite{CN}]
The generating function of the topological recursion data
\be
S=\left( \frac{1}{2}z^2, z^{2\alpha-1}+\sum_{k=2\alpha}^\infty y_k z^k\right)
\ee
is obtained via translation of the appropriate tau-function
\be\label{transl}
\tau_\alpha\left(\left\{ t_{2k+1}-\frac{1}{\hbar}\frac{y_{2k-1}}{2k+1}\right\}\right).
\ee
\end{proposition}
Note that only $t_{2k+1}$ with $k\geq \alpha+1$ are translated.
\begin{remark}
Actually, the proposition of Chekhov--Norbury is more general and describes $y=\sum_{k=2\alpha-1}^\infty y_k z^k$ with $y_{2\alpha-1}\neq 1$. 
We do not need this more general version here.
\end{remark}

\begin{corollary}\label{cor11}
The functions $\tau_\alpha({\bf t},{\bf s^\alpha})$ in Theorem \ref{taueq} are described by the topological recursion on the genus zero spectral curves $S_\alpha$ with
\be\label{SS1}
S_0=\left(\frac{1}{2}z^2, z^{-1}+\sum_{k=1}^\infty u^{2k} z^{2k-1}\right)=\left(\frac{1}{2}z^2,\frac{1}{z(1-u^2z^2)}\right)
\ee
and
\be\label{Scur}
S_1=\left( \frac{1}{2}z^2, z+\sum_{k=1}^\infty u^{2k} z^{2k+1}\sum_{j=0}^k\frac{1}{2j+1}\right)=
\left( \frac{1}{2}z^2,\frac{1}{2u(1-u^2 z^2)}\log\frac{1+u z}{1-u z}\right).
\ee
\end{corollary}
The stable correlators of topological recursion for these curves are related to the functions \eqref{corsh} by
\be
\omega_{g,n}^{\alpha,u}=\hbox{d}_1\dots \hbox{d}_n \tilde{W}_{g,n}^\alpha(z_1,\dots,z_n).
\ee

Spectral curves in Corollary \ref{cor11} are described by the equations
\be
x(1-2u^2x)^2 y^2=\frac{1}{2}
\ee
and
\be
u^{-2}\left(1-e^{2u(1-2u^2x)y}\right)^2=2x\left(1+e^{2u(1-2u^2x)y}\right)^2
\ee
respectively.
Locally at $x=0$ their behavior is described by the Bessel and Airy curves correspondingly. For both curves the only zero of $\hbox{d} x =z \hbox{d} z$ is simple and is located at $z=0$. The curve for $\alpha=1$ is not algebraic, similar to the case of the Weil--Petersson volumes \cite{EO1}.

\subsection{Topological recursion for triple Hodge integrals}\label{STR}
In this section, we will focus on the case $\alpha=1$, the case  $\alpha=0$ is considered in Section \ref{S64}.
Topological recursion for the triple Hodge integrals of the form \eqref{Intin} for $\alpha=1$ is described by Bouchard--Mari\~{n}o conjecture \cite{BM}, which was proven by Chen and Zhou 
\cite{Chen,Zhou1}. In this section we remind the reader these results basically following \cite{Zhou1}. 

To formulate the results of Chen and Zhou let us describe the relation between the combination of the Hodge classes used in these papers and $\Lambda_g(p)$,
\be
\Lambda_g^\vee(a)=a^g\Lambda_g\left(-\frac{1}{a}\right).
\ee
Let
\be\label{newphik}
\phi_b(z)=\left(X\frac{\p}{\p X}\right)^b \frac{\frac{1}{z}-(a+1)}{(a+1)^2},
\ee
where
\be\label{X}
X=\frac{a^a}{(a+1)^{a+1}}\left(1-(a+1)z\right)\left(1+\frac{(a+1)z}{a}\right)^a.
\ee
Then $X$ and
\be
Y=\frac{a}{a+1}+z
\ee
satisfy the spectral curve equation 
\be
X=Y^a-Y^{a+1}.
\ee

Consider the topological recursion for the genus zero spectral curve
\be\label{TRi}
S=\left(\log X,\log Y\right)
\ee
and
\begin{align}
	& \omega_{0,1}(z_1) =  \log Y(z_1) \, \hbox{d} \log X (z_1) ; 
	\\ \notag
	& \omega_{0,2}(z_1,z_2) = B_C(z_1,z_2).
\end{align}
Then the stable correlators are related to the triple Hodge integrals satisfying the Calabi--Yau condition by  \cite{Chen,Zhou1}
\begin{multline}\label{oldcf}
\omega_{g,n}(z_1,\dots,z_n)\\=(-1)^n (a(a+1))^{g+n-1}\sum_{a_1,\dots,a_n\ge 0} \int_{\overline{\mathcal{M}}_{g,n}}
 \Lambda_g(-1)\Lambda_g\left(-\frac{1}{a}\right)\Lambda_g\left(\frac{1}{a+1}\right)\prod_{j=1}^n \psi_j^{a_j}
\hbox{d} \phi_{a_i}(z_i).
\end{multline}

Under the {\em rescaling} $y \mapsto \epsilon y$  the correlators of the topological recursion change as
\be
\omega_{g,n} \mapsto \epsilon^{2-2g-n} \omega_{g,n}.
\ee
Let us consider this transformation with $\epsilon=-\sqrt{a(a+1)}$ for the spectral curves \eqref{TRi}, then for the new spectral curve $S=\left(\log X,\log \tilde{Y}\right)$ the stable correlators \eqref{oldcf} acquire the form
\be\label{om}
\tilde{\omega}_{g,n}(z_1,\dots,z_n)=\sum_{a_1,\dots,a_n\ge 0} \int_{\overline{\mathcal{M}}_{g,n}}
 \Lambda_g(-1)\Lambda_g\left(-\frac{1}{a}\right)\Lambda_g\left(\frac{1}{a+1}\right)\prod_{j=1}^n \psi_j^{a_j}
\hbox{d} \tilde{\phi}_{a_i}(z_i).
\ee
Here  $\tilde{Y}=(Y)^{-\sqrt{a(a+1)}}$ and
\begin{equation}
\begin{split}
\tilde{\phi}_b(z)&=\sqrt{a(a+1)}{\phi}_b(z)\\
&=\left(X\frac{\p}{\p X}\right)^b 
\frac{\sqrt{a}}{(a+1)^\frac{3}{2}}\left(\frac{1}{z}-(a+1)\right).
\end{split}
\end{equation}

To identify the functions $\tilde{\phi}_k(z)$ with the coefficients $\Phi_{k}(z)$ in \eqref{phik} let us make the change of the variable 
\be
\frac{\sqrt{a}}{(a+1)^\frac{3}{2}}\left(\frac{1}{z}-(a+1)\right) \mapsto \frac{1}{z}.
\ee
Then $X(z)$ and $\tilde{Y}(z)$ get the form
\be\label{bigX}
X(z)=\frac{a^a}{(a+1)^{a+1}}\frac{\left(1+\sqrt{\frac{a+1}{a}}z\right)^a}{\left(1+\sqrt{\frac{a}{a+1}}z\right)^{a+1}}, \quad \quad \tilde{Y}=\left(\frac{a}{a+1}\frac{1+\sqrt{\frac{a+1}{a}}z}{1+\sqrt{\frac{a}{a+1}}z}\right)^{-\sqrt{a(a+1)}}.
\ee

To identify the triple Hodge class of \eqref{om} with the one in \eqref{FF} one should consider $q,p \in\left\{1,\frac{1}{a},-\frac{1}{a+1}\right\}$. Put $q=1$, $p=-\frac{1}{a+1}$ or $p=\frac{1}{a}$. Then \eqref{fpq} and \eqref{y1} are given by
\begin{equation}
\begin{split}
{\mathtt x}(z)&=(a+1)\log\left(1+\sqrt{\frac{a}{a+1}}z\right)-a\log\left(1+\sqrt{\frac{a+1}{a}}z\right),\\
{\mathtt y}_1(z)&=\sqrt{a(a+1)}\left(\log\left(1+\sqrt{\frac{a+1}{a}z}\right)-\log\left(1+\sqrt{\frac{a}{a+1}z}\right)\right).
\end{split}
\end{equation}
This pair is symplectically equivalent to the pair $\log X(z)$ and $\log \tilde{Y}(z)$,
\be
-\log X(z)= {\mathtt x}(z)-\log\left(\frac{a^a}{(a+1)^{a+1}}\right),\quad \quad -\log\tilde{Y}(z)={\mathtt y}_1(z)-\sqrt{a(a+1)}\log \frac{a+1}{a},
\ee
that is $\hbox{d}\log X \wedge  \hbox{d} \log\tilde{Y}=\hbox{d}  {\mathtt x}\wedge \hbox{d}  {\mathtt y}_1$. It is easy to see that for $\left(\log X,\log \tilde{Y}\right)$ and  $({\mathtt x},{\mathtt y}_1)$ all stable correlators of topological recursion coincide.
In particular, 
\be
\left(X\frac{\p}{\p X}\right)^b \frac{1}{z} = \left(-\frac{\p}{\p {\mathtt x}}\right)^k\cdot \frac{1}{z}.
\ee 
and the functions $\Phi_k(z)$ and $\phi_k(z)$, given \eqref{phik} and \eqref{newphik} coincide with each other.

For arbitrary $p$ and $q$ let us consider
\be\label{gencc}
\omega_{g,n}^{(\alpha),p,q}(z_1,\dots,z_n)=\hbox{d}_1\dots \hbox{d}_n\sum_{a_1,\dots,a_n\ge 0} \int_{\overline{\mathcal{M}}_{g,n}}
 \Lambda_g(-p)\Lambda_g\left(-q\right)\Lambda_g\left(\frac{pq}{p+q}\right)\prod_{j=1}^n \psi_j^{a_j}
\Phi_{a_i}(z_i),
\ee
where $\Phi_{a}$'s are given by \eqref{phik}.
If we transform $p\mapsto \epsilon^2 p$, $q\mapsto \epsilon^2 q$ and $z\mapsto \epsilon^{-1}z$, then ${\mathtt x}\mapsto \epsilon^{-2}{\mathtt x}$ and ${\mathtt y}_\alpha\mapsto \epsilon^{1-2\alpha}{\mathtt y}_{\alpha}$, therefore $\omega_{g,n}^{(\alpha),p,q}(z_1,\dots,z_n) \mapsto \epsilon^{(2\alpha+1)(2g-2+n)}\omega_{g,n}^{(\alpha),p,q}(z_1,\dots,z_n)$. Using this transformation for $\alpha=1$ we can substitute the values $q_0=1$ and $p_0=-\frac{1}{a+1}$ or $p_0=\frac{1}{a}$ in \eqref{om}
by arbitrary $q$ and $p$. Therefore, the tau-function $\tau^{(1)}_{q,p}$ and correlators $\omega_{g,n}^{(1),p,q}$ are governed by topological recursion with $S=({\mathtt x},{\mathtt y}_1)$. This spectral curve is given by the equation \eqref{spp}.


Let us put $p=-2q=2u^2$. Then for $\alpha=1$
 the triple Hodge classes in \eqref{gencc} coincide with the triple Hodge classes in \eqref{corf}, and for the stable cases with $n>0$ we have
\be
\omega_{g,n}^{(1),2u^2,-u^2}(z_1,\dots,z_n)=d_1\dots d_n W_{g,n}^1(z_1,\dots,z_n).
\ee
We have
\begin{equation}
\begin{split}
{\mathtt x}&=-\frac{1}{2u^2}\log\left(1-u^2z^2\right),\\
{\mathtt y}_1&=\frac{1}{2u}\log\frac{1+uz}{1-uz}.
\end{split}
\end{equation}

If we compare the curve  $S=({\mathtt x},{\mathtt y}_1)$ to the curve $S_1$ in \eqref{Scur} with $x= \frac{1}{2}z^2$ and $y=\frac{1}{2u(1-u^2 z^2)}\log\frac{1+u z}{1-u z}$, we have
\be
{\mathtt x}=g({x}), \quad \quad {\mathtt y}_1=\frac{1}{g'({x})}{y},
\ee
where $x=\frac{z^2}{2}$ and
\be\label{gfun}
g(x)=-\frac{1}{2u^2}\log(1-2u^2x).
\ee
We see that the curve $({\mathtt x},{\mathtt y}_1)$ is related to $({x},{y})$ by a symplectic transformation of the form \eqref{symptr}. 

Now from Proposition \ref{Props} it follows that stable correlators for the topological recursion data $S=({\mathtt x},{\mathtt y}_1)$ and \eqref{Scur} coincide.
This provides an independent proof of Corollary \ref{cor} for $\alpha=1$.

\subsection{Topological recursion for triple \texorpdfstring{$\Theta$}--Hodge integrals}\label{S64}

In this section, we will consider the Chekhov--Eynard--Orantin topological recursion for the triple $\Theta$-Hodge integrals satisfying the Calabi--Yau condition ($\alpha=0$). Namely, for arbitrary $q$ and $p$ we prove an analog of the Bouchard--Mari\~{n}o topological recursion for this case.

Let us consider the triple $\Theta$-Hodge integrals satisfying the Calabi--Yau condition for arbitrary $q$ and $p$. For $\omega_{g,n}^{(0),p,q}(z_1,\dots,z_n)$, given by \eqref{gencc}, we have
\begin{proposition}
Stable correlators  $\omega_{g,n}^{(0),p,q}$ are given by topological recursion for 
\begin{equation}
\begin{split}
S&=({\mathtt x},{\mathtt y}_0)\\
&=\left(\frac{p+q}{pq}\log\left(1+\frac{qz}{\sqrt{p+q}}\right)-\frac{1}{p}\log(1+\sqrt{p+q}z),\frac{1}{z}\right).
\end{split}
\end{equation}
\end{proposition}
\begin{proof}
The statement follows from the results of Chekhov and Norbury \cite{CN} and almost literally repeats the proof of Theorem 12 in \cite{NP1}. Let us give only a brief outline of the proof. The generating function of the triple Hodge integrals $Z^{(1)}_{p,q}$ given by \eqref{GenfN} can be described by Givental's formalism. For $Z^{(0)}_{p,q}$, according to \cite[Proposition 3.9]{NP1}, we have a similar looking Givental description with $\tau_{1}$ substituted by $\tau_{0}$ and a different translation operator. The details of both Givental's descriptions are given in  \cite[Section 3]{H3_1}. 

By \cite{DOSS, CN} Givental's description is related to the topological recursion. The topological recursion for $\tau_{p,q}^{(1)}$ is described by $S=({\mathtt x},{\mathtt y}_1)$, see Section \ref{STR}. Then, according to \cite{NP1}, the $\Theta$-version of the same tau-function is described by substitution of ${\mathtt y}_1$ with $\frac{\p {\mathtt y}_1}{\p {\mathtt x}}$. From \eqref{ygr} we have
\be
\frac{\p {\mathtt y}_1}{\p {\mathtt x}}={\mathtt y}_0
\ee 
which completes the proof.

\end{proof}

For for $p=-2q=2u^2$ the stable correlators are related to \eqref{corf} with $\alpha =0$, 
\begin{equation}
\begin{split}\label{corn}
\omega_{g,n}^{(0),2u^2,-u^2}(z_1,\dots,z_n)&=d_1\dots d_n W_{g,n}^0(z_1,\dots,z_n)\\
\end{split}
\end{equation}
Then these correlators are given by topological recursion:
\begin{corollary}
The stable correlators $\omega_{g,n}^{(0),2u^2,-u^2}$ are given by the topological recursion with
\be\label{Sn}
S=\left(-\frac{1}{2u^2}\log(1-u^2 z^2), \frac{1}{z}\right).
\ee
\end{corollary}
\begin{proof}
This corollary has an independent proof, based on the symplectic invariance.
From Corollary \ref{cor} it follows that the stable correlators $\omega_{g,n}^{(0),2u^2,-u^2}$ are equal to $d_1\dots d_n \tilde{W}^0_{g,n}$. From Corollary \ref{cor11} these correlators are given by the topological recursion with the data 
\be\label{SS3}
S=\left(\frac{1}{2}z^2,\frac{1}{z(1-u^2z^2)}\right).
\ee
These $x$ and $y$ are symplectically equivalent to \eqref{Sn} for the symplectic transformation \eqref{symptr} with $g(x)$ given by \eqref{gfun},
therefore, the statement follows from Proposition \ref{Props}.
\end{proof}
The curve \eqref{Sn} is described by the equation
\be
u^{-2}y^2(1-e^{-2u^2 x})=1,
\ee
which for $u=0$ reduces to the Bessel curve $2xy^2=1$.

\appendix
\section{Topological expansion of  \texorpdfstring{$\log\tau_{q,p}^{(\alpha)}$}- } \label{AA}

\begin{flalign}
  \begin{aligned}
F_1^0&=\frac{1}{8}t_1,
\end{aligned}&&&
\end{flalign}
\begin{flalign}
  \begin{aligned}
F_2^0&=\frac{1}{16}\,{t_{{1}}}^{2}
-{\frac {1}{128}}\,{\frac {{p}^{2}+pq+{q}^{2}}{p+q}},
\end{aligned}&&&
\end{flalign}
\begin{flalign}
  \begin{aligned}
F_3^0&={\frac {9}{128}}\,t_{{3}}+\frac{1}{24}\,{t_{{1}}}^{3}
+{\frac {3}{64}}\,{\frac { \left( p+2\,q \right) t_{{2}}}{\sqrt {p+q}}}
-{\frac {1}{128}}\,{\frac { \left( 2\,{p}^{2}-pq-{q}^{2} \right) t_{{1
}}}{p+q}}
\end{aligned}&&&
\end{flalign}
\begin{flalign}
  \begin{aligned}
 F_4^0&={\frac {27}{128}}\,t_{{3}}t_{{1}}+\frac{1}{32}\,{t_{{1}}}^{4}
+{\frac {9}{64}}\,{\frac { \left( p+2\,q \right) t_{{1}}t_{{2}}}{\sqrt 
{p+q}}}
-{\frac {3}{128}}\,{\frac { \left( {p}^{2}-2\,pq-2\,{q}^{2} \right) {t
_{{1}}}^{2}}{p+q}}\\
&+{\frac {1}{512}}\,{\frac {{p}^{4}+2\,{p}^{3}q+3\,{p}^{2}{q}^{2}+2\,p{q
}^{3}+{q}^{4}}{ \left( p+q \right) ^{2}}}
\end{aligned}&&&
\end{flalign}
\begin{flalign}
  \begin{aligned}
F_5^0&={\frac {225}{1024}}\,t_{{5}}+{\frac {27}{64}}\,t_{{3}}{t_{{1}}}^{2}+\frac{1}{40}\,{t_{{1}}}^{5}
+{\frac {75}{256}}\,{\frac { \left( p+2\,q \right) t_{{4}}}{\sqrt {p+q}
}}+{\frac {9}{32}}\,{\frac { \left( p+2\,q \right) t_{{2}}{t_{{1}}}^{2
}}{\sqrt {p+q}}}\\
&+{\frac {3}{512}}\,{\frac {t_{{3}} \left( 4\,{p}^{2}+79\,pq+79\,{q}^{2}
 \right) }{p+q}}-{\frac {1}{64}}\,{\frac { \left( 2\,{p}^{2}-7\,pq-7\,
{q}^{2} \right) {t_{{1}}}^{3}}{p+q}}\\
&-{\frac {1}{512}}\,{\frac { \left( 22\,{p}^{3}+21\,{p}^{2}q-69\,p{q}^{
2}-46\,{q}^{3} \right) t_{{2}}}{ \left( p+q \right) ^{3/2}}}\\
&+{\frac {1}{1024}}\,{\frac { \left( 8\,{p}^{4}-6\,{p}^{3}q-5\,{p}^{2}{q
}^{2}+2\,p{q}^{3}+{q}^{4} \right) t_{{1}}}{ \left( p+q \right) ^{2}}}
\end{aligned}&&&
\end{flalign}
\begin{flalign}
  \begin{aligned}
F_6^0&={\frac {1125}{1024}}\,t_{{5}}t_{{1}}+{\frac {567}{1024}}\,{t_{{3}}}^{2
}+{\frac {45}{64}}\,t_{{3}}{t_{{1}}}^{3}+\frac{1}{48}\,{t_{{1}}}^{6}
+{\frac {375}{256}}\,{\frac { \left( p+2\,q \right) t_{{1}}t_{{4}}}{
\sqrt {p+q}}}\\
&+{\frac {189}{256}}\,{\frac { \left( p+2\,q \right) t_{{2
}}t_{{3}}}{\sqrt {p+q}}}+{\frac {15}{32}}\,{\frac { \left( p+2\,q
 \right) t_{{2}}{t_{{1}}}^{3}}{\sqrt {p+q}}}
+{\frac {3}{256}}\,{\frac { \left( 10\,{p}^{2}+229\,pq+229\,{q}^{2}
 \right) t_{{3}}t_{{1}}}{p+q}}\\
 &+{\frac {63}{256}}\,{\frac { \left( {p}^
{2}+4\,pq+4\,{q}^{2} \right) {t_{{2}}}^{2}}{p+q}}
-{\frac {5}{128}}\,{
\frac { \left( {p}^{2}-5\,pq-5\,{q}^{2} \right) {t_{{1}}}^{4}}{p+q}}\\
&
-{\frac {1}{512}}\,{\frac { \left( 110\,{p}^{3}-21\,{p}^{2}q-723\,p{q}
^{2}-482\,{q}^{3} \right) t_{{2}}t_{{1}}}{ \left( p+q \right) ^{3/2}}}\\
&+{\frac {1}{512}}\,{\frac { \left( 10\,{p}^{4}-35\,{p}^{3}q-11\,{p}^{2}
{q}^{2}+48\,p{q}^{3}+24\,{q}^{4} \right) {t_{{1}}}^{2}}{ \left( p+q
 \right) ^{2}}}\\
&-{\frac {1}{12288}}\,{\frac {17\,{p}^{6}+51\,{p}^{5}q+105\,{p}^{4}{q}^
{2}+125\,{p}^{3}{q}^{3}+105\,{p}^{2}{q}^{4}+51\,p{q}^{5}+17\,{q}^{6}}{
 \left( p+q \right) ^{3}}}
\end{aligned}&&&
\end{flalign}
\begin{flalign}
  \begin{aligned}
F_7^0&=
{\frac {55125\,t_{{7}}}{32768}}+{\frac {3375}{1024}}\,t_{{5}}{t_{{1}}}
^{2}+{\frac {1701}{512}}\,{t_{{3}}}^{2}t_{{1}}+{\frac {135}{128}}\,t_{
{3}}{t_{{1}}}^{4}+{\frac {1}{56}}\,{t_{{1}}}^{7}+{\frac {55125}{16384}}\,{\frac { \left( p+2\,q \right) t_{{6}}}{\sqrt 
{p+q}}}\\
&+{\frac {1125}{256}}\,{\frac { \left( p+2\,q \right) t_{{4}}{t_
{{1}}}^{2}}{\sqrt {p+q}}}+{\frac {567}{128}}\,{\frac { \left( p+2\,q
 \right) t_{{3}}t_{{2}}t_{{1}}}{\sqrt {p+q}}}+{\frac {45}{64}}\,{
\frac { \left( p+2\,q \right) t_{{2}}{t_{{1}}}^{4}}{\sqrt {p+q}}}\\
&+{\frac {1875}{32768}}\,{\frac { \left( 26\,{p}^{2}+173\,pq+173\,{q}^{2
} \right) t_{{5}}}{p+q}}+{\frac {9}{512}}\,{\frac { \left( 20\,{p}^{2}
+521\,pq+521\,{q}^{2} \right) t_{{3}}{t_{{1}}}^{2}}{p+q}}\\
&+{\frac {189}
{128}}\,{\frac { \left( {p}^{2}+4\,pq+4\,{q}^{2} \right) {t_{{2}}}^{2}
t_{{1}}}{p+q}}-{\frac {3}{128}}\,{\frac { \left( 2\,{p}^{2}-13\,pq-13
\,{q}^{2} \right) {t_{{1}}}^{5}}{p+q}}\\
&-{\frac {25}{4096}}\,{\frac { \left( 67\,{p}^{3}-387\,{p}^{2}q-1563\,p
{q}^{2}-1042\,{q}^{3} \right) t_{{4}}}{ \left( p+q \right) ^{3/2}}}\\
&-{
\frac {3}{512}}\,{\frac { \left( 110\,{p}^{3}-147\,{p}^{2}q-1101\,p{q}
^{2}-734\,{q}^{3} \right) t_{{2}}{t_{{1}}}^{2}}{ \left( p+q \right) ^{
3/2}}}\\
&-{\frac {3}{32768}}\,{\frac { \left( 1548\,{p}^{4}+9796\,{p}^{3}q-7681
\,{p}^{2}{q}^{2}-34954\,p{q}^{3}-17477\,{q}^{4} \right) t_{{3}}}{
 \left( p+q \right) ^{2}}}\\
 &+{\frac {1}{1024}}\,{\frac { \left( 40\,{p}^
{4}-250\,{p}^{3}q+63\,{p}^{2}{q}^{2}+626\,p{q}^{3}+313\,{q}^{4}
 \right) {t_{{1}}}^{3}}{ \left( p+q \right) ^{2}}}\\
&+{\frac {1}{16384}}\,{\frac { \left( 1052\,{p}^{5}+308\,q{p}^{4}-4561\,
{q}^{2}{p}^{3}-284\,{q}^{3}{p}^{2}+4135\,{q}^{4}p+1654\,{q}^{5}
 \right) t_{{2}}}{ \left( p+q \right) ^{5/2}}}\\
&-{\frac {1}{32768}}\,{\frac { \left( 272\,{p}^{6}-236\,{p}^{5}q-176\,{
p}^{4}{q}^{2}+115\,{p}^{3}{q}^{3}+45\,{p}^{2}{q}^{4}-15\,p{q}^{5}-5\,{
q}^{6} \right) t_{{1}}}{ \left( p+q \right) ^{3}}}
\end{aligned}&&&
\end{flalign}

\begin{flalign}
  \begin{aligned}
F_1^1&=\frac{1}{8}t_{{3}}+\frac{1}{6}{t_{{1}}}^{3}+\frac{1}{12}{\frac { \left( p+2\,q
 \right) t_{{2}}}{\sqrt {p+q}}}-\frac{1}{24}{\frac {{p}^{2}t_{{1}}}{p+q}},
\end{aligned}&&&
\end{flalign}
\begin{flalign}
  \begin{aligned}
 F_2^1&=
\frac{1}{2}\,t_{{3}}{t_{{1}}}^{3}+\frac{5}{8}\,t_{{5}}t_{{1}}+\frac{3}{16}\,{t_{{3}}}^{2}\\
&+\frac{5}{6}\,{\frac { \left( p+2\,q \right) t_{{4}}t_{{1}}}{\sqrt {p+q}}}+\frac{1}{4}
\,{\frac { \left( p+2\,q \right) t_{{3}}t_{{2}}}{\sqrt {p+q}}}+\frac{1}{3}\,{
\frac { \left( p+2\,q \right) t_{{2}}{t_{{1}}}^{3}}{\sqrt {p+q}}}\\
&+\frac{1}{8}\,{\frac { \left( {p}^{2}+12\,pq+12\,{q}^{2} \right) t_{{3}}t_{{1}}
}{p+q}}+\frac{1}{12}\,{\frac { \left( {p}^{2}+4\,pq+4\,{q}^{2} \right) {t_{{2}
}}^{2}}{p+q}}+\frac{1}{6}\,q{t_{{1}}}^{4}\\
&-\frac{1}{12}\,{\frac { \left( {p}^{3}-{p}^{2}q-9\,p{q}^{2}-6\,{q}^{3}
 \right) t_{{1}}t_{{2}}}{ \left( p+q \right) ^{3/2}}}
-\frac{1}{48}\,{\frac {q \left( 2\,{p}^{2}-pq-{q}^{2} \right) {t_{{1}}}^{2}}{p
+q}}
-{\frac {1}{5760}}\,{\frac {{p}^{2}{q}^{2}}{p+q}},
\end{aligned}&&&
\end{flalign}

\begin{flalign}
  \begin{aligned}
 F_3^1&=\frac{5}{8}\,t_{{5}}{t_{{1}}}^{4}+\frac{3}{2}\,{t_{{3}}}^{2}{t_{{1}}}^{3}+{\frac {35}{
16}}\,t_{{7}}{t_{{1}}}^{2}+{\frac {15}{4}}\,t_{{5}}t_{{3}}t_{{1}}+\frac{3}{8}
\,{t_{{3}}}^{3}+{\frac {105}{128}}\,t_{{9}}\\
&+{\frac {35}{8}}\,{\frac { \left( p+2\,q \right) t_{{6}}{t_{{1}}}^{2}}{
\sqrt {p+q}}}+\frac{3}{4}\,{\frac { \left( p+2\,q \right) {t_{{3}}}^{2}t_{{2}}
}{\sqrt {p+q}}}+{\frac {35}{16}}\,{\frac { \left( p+2\,q \right) t_{{8
}}}{\sqrt {p+q}}}+\frac{5}{6}\,{\frac { \left( p+2\,q \right) t_{{4}}{t_{{1}}}
^{4}}{\sqrt {p+q}}}\\
&+2\,{\frac { \left( p+2\,q \right) t_{{3}}t_{{2}}{t
_{{1}}}^{3}}{\sqrt {p+q}}}+\frac{5}{2}\,{\frac { \left( p+2\,q \right) t_{{5}}
t_{{2}}t_{{1}}}{\sqrt {p+q}}}+5\,{\frac { \left( p+2\,q \right) t_{{4}
}t_{{3}}t_{{1}}}{\sqrt {p+q}}}\\
&+\frac{10}{3}\,{\frac { \left( {p}^{2}+4\,pq+4\,{q}^{2} \right) t_{{4}}t_{{2}}t
_{{1}}}{p+q}}+{\frac {9}{8}}\,{\frac { \left( {p}^{2}+8\,pq+8\,{q}^{2}
 \right) {t_{{3}}}^{2}t_{{1}}}{p+q}}\\
 &+{\frac {5}{48}}\,{\frac { \left( 
23\,{p}^{2}+140\,pq+140\,{q}^{2} \right) t_{{5}}{t_{{1}}}^{2}}{p+q}}+\frac{1}{4}
\,{\frac { \left( {p}^{2}+10\,pq+10\,{q}^{2} \right) t_{{3}}{t_{{1}}
}^{4}}{p+q}}\\
&+{\frac {7}{288}}\,{\frac { \left( 76\,{p}^{2}+391\,pq+391
\,{q}^{2} \right) t_{{7}}}{p+q}}+\frac{2}{3}\,{\frac { \left( {p}^{2}+4\,pq+4
\,{q}^{2} \right) {t_{{2}}}^{2}{t_{{1}}}^{3}}{p+q}}\\
&+\frac{1}{2}\,{\frac {
 \left( {p}^{2}+4\,pq+4\,{q}^{2} \right) t_{{3}}{t_{{2}}}^{2}}{p+q}}
 +\frac{1}{2}\,{\frac { \left( {p}^{3}+17\,{p}^{2}q+45\,p{q}^{2}+30\,{q}^{3}
 \right) t_{{3}}t_{{2}}t_{{1}}}{ \left( p+q \right) ^{3/2}}}\\
&+\frac{1}{12}\,{
\frac { \left( {p}^{3}+79\,{p}^{2}q+231\,p{q}^{2}+154\,{q}^{3}
 \right) t_{{4}}{t_{{1}}}^{2}}{ \left( p+q \right) ^{3/2}}}\\
 &+{\frac {11
}{12}}\,{\frac { \left( p+2\,q \right) qt_{{2}}{t_{{1}}}^{4}}{\sqrt {p
+q}}}+\frac{1}{9}\,{\frac { \left( {p}^{3}+6\,{p}^{2}q+12\,p{q}^{2}+8\,{q}^{3}
 \right) {t_{{2}}}^{3}}{ \left( p+q \right) ^{3/2}}}\\
 &+{\frac {7}{192}}
\,{\frac { \left( 10\,{p}^{3}+167\,{p}^{2}q+441\,p{q}^{2}+294\,{q}^{3}
 \right) t_{{6}}}{ \left( p+q \right) ^{3/2}}}\\
&-\frac{1}{16}\,{\frac { \left( 2\,{p}^{4}-2\,{p}^{3}q-101\,{p}^{2}{q}^{2}-198
\,p{q}^{3}-99\,{q}^{4} \right) t_{{3}}{t_{{1}}}^{2}}{ \left( p+q
 \right) ^{2}}}+{\frac {5}{24}}\,{q}^{2}{t_{{1}}}^{5}\\
 &-{\frac {1}{384}}
\,{\frac { \left( 57\,{p}^{4}-236\,{p}^{3}q-2769\,{p}^{2}{q}^{2}-5066
\,p{q}^{3}-2533\,{q}^{4} \right) t_{{5}}}{ \left( p+q \right) ^{2}}}\\
&-\frac{1}{6}\,{\frac { \left( {p}^{4}-2\,{p}^{3}q-26\,{p}^{2}{q}^{2}-48\,p{q}^{3
}-24\,{q}^{4} \right) {t_{{2}}}^{2}t_{{1}}}{ \left( p+q \right) ^{2}}}\\
&-{\frac {1}{1440}}\,{\frac { \left( 37\,{p}^{5}+520\,{p}^{4}q-178\,{p}
^{3}{q}^{2}-5172\,{p}^{2}{q}^{3}-7580\,p{q}^{4}-3032\,{q}^{5} \right) 
t_{{4}}}{ \left( p+q \right) ^{5/2}}}\\
&-\frac{1}{24}\,{\frac { \left( 7\,{p}^{3}
-4\,{p}^{2}q-54\,p{q}^{2}-36\,{q}^{3} \right) qt_{{2}}{t_{{1}}}^{2}}{
 \left( p+q \right) ^{3/2}}}\\
&-{\frac {1}{144}}\,{\frac { \left( 9\,{p}^{2}-8\,pq-8\,{q}^{2}
 \right) {q}^{2}{t_{{1}}}^{3}}{p+q}}\\
 &+{\frac {1}{960}}\,{\frac {
 \left( 7\,{p}^{6}-16\,q{p}^{5}-239\,{q}^{2}{p}^{4}-170\,{p}^{3}{q}^{3
}+605\,{q}^{4}{p}^{2}+828\,p{q}^{5}+276\,{q}^{6} \right) t_{{3}}}{
 \left( p+q \right) ^{3}}}\\
&+{\frac {1}{2880}}\,{\frac { \left( 21\,{p}^{5}-2\,{p}^{4}q-120\,{p}^{3
}{q}^{2}-40\,{p}^{2}{q}^{3}+60\,p{q}^{4}+24\,{q}^{5} \right) qt_{{2}}
}{ \left( p+q \right) ^{5/2}}}\\
&+{\frac {7}{5760}}\,{\frac {{p}^{4}{q}^{2}t_{{1}}}{ \left( p+q \right) 
^{2}}}.
\end{aligned}&&&
\end{flalign}

\bibliographystyle{alphaurl}
\bibliography{KPTRref}

\end{document}